\documentclass{amsart} 

\usepackage{amssymb}
\usepackage{amsmath}
\usepackage{amsthm}
\usepackage{color}

\usepackage{enumerate} 

\theoremstyle{plain}
 \newtheorem{theorem}{Theorem}[section]
 \newtheorem{main-theorem}{Theorem}
 \newtheorem{corollary*}[main-theorem]{Corollary}
 \newtheorem{lemma}[theorem]{Lemma}
 \newtheorem{corollary}[theorem]{Corollary}
 \newtheorem{proposition}[theorem]{Proposition}

\theoremstyle{definition}
 \newtheorem{remark}[theorem]{Remark}
 \newtheorem{example}[theorem]{Example}
\newtheorem{notation}[theorem]{Notation}

\renewcommand{\mod}{\operatorname{mod}}

\newcommand{\alg}{\operatorname{alg}}
\newcommand{\rad}{\operatorname{rad}}
\newcommand{\soc}{\operatorname{soc}}
\newcommand{\topp}{\operatorname{top}}
\newcommand{\umod}{\operatorname{\underline{mod}}}

\newcommand{\Ext}{\operatorname{Ext}}
\newcommand{\Hom}{\operatorname{Hom}}
\newcommand{\Ker}{\operatorname{Ker}}

\newcommand{\Tr}{\operatorname{Tr}}
\newcommand{\GL}{\operatorname{GL}}
\newcommand{\op}{\operatorname{op}}
\newcommand{\id}{\operatorname{id}}

\newcommand{\intt}{\operatorname{int}}

\newcommand{\bA}{\mathbb{A}}
\newcommand{\bD}{\mathbb{D}}
\newcommand{\bE}{\mathbb{E}}

\newcommand{\bN}{\mathbb{N}}
\newcommand{\bP}{\mathbb{P}}

\newcommand{\bR}{\mathbb{R}}
\newcommand{\bS}{\mathbb{S}}
\newcommand{\bT}{\mathbb{T}}
\newcommand{\bZ}{\mathbb{Z}}

\newcommand{\cA}{\mathcal{A}}
\newcommand{\cB}{\mathcal{B}}
\newcommand{\calD}{\mathcal{D}}
\newcommand{\cK}{\mathcal{K}}
\newcommand{\cO}{\mathcal{O}}

\newcommand{\cT}{\mathcal{T}}

\newcommand{\ba}{\bar{\alpha}}
\newcommand{\overbar}[1]{\mkern 5mu\overline{\mkern-5mu#1\mkern-5mu}\mkern 5mu}

\usepackage{etex} 

\usepackage{dsfont}         
\usepackage[h]{esvect}

\usepackage{xy}

\usepackage{pgf}
\usepackage{tikz}
\usepackage{xcolor}

\usetikzlibrary{arrows,shapes,automata,backgrounds,petri,calc}

\usepackage{xypic}

\newcommand{\tikzAngleOfLine}{\tikz@AngleOfLine}
\def\tikz@AngleOfLine(#1)(#2)#3{%
\pgfmathanglebetweenpoints{%
\pgfpointanchor{#1}{center}}{%
\pgfpointanchor{#2}{center}}
\pgfmathsetmacro{#3}{\pgfmathresult}%
}

\begin{document}

\title{Algebras of generalized dihedral type}

%
%
%

{\def\thefootnote{}
\footnote{The authors gratefully acknowledge support from the research grant
DEC-2011/02/A/ST1/00216 of the National Science Center Poland.}
}

\author[K. Edrmann]{Karin Erdmann}
\address[Karin Erdmann]{Mathematical Institute,
   Oxford University,
   ROQ, Oxford OX2 6GG,
   United Kingdom}
\email{erdmann@maths.ox.ac.uk}

\author[A. Skowro\'nski]{Andrzej Skowro\'nski}
\address[Andrzej Skowro\'nski]{Faculty of Mathematics and Computer Science,
   Nicolaus Copernicus University,
   Chopina~12/18,
   87-100 Toru\'n,
   Poland}
\email{skowron@mat.uni.torun.pl}

\begin{abstract}
We provide a complete classification of all algebras
of generalized dihedral type,
which are 
natural generalizations of algebras which occurred
in the study of blocks of group algebras
with dihedral defect groups. 
This gives a description by quivers and relations coming from
surface triangulations.
\end{abstract}

\maketitle

\section{Introduction and main results}\label{sec:intro}

The representation theory of associative finite-dimensional  algebras 
over an algebraically closed field $K$ is divided 
into two disjoint classes, by 
the Tame and Wild Theorem of Drozd
(see \cite{CB1}, \cite{Dr}).
The first class consists of the \emph{tame algebras}
for which the indecomposable modules occur in each dimension $d$
in a finite number of discrete and a finite number of one-parameter
families.
The second class is formed by the \emph{wild algebras}
whose representation theory comprises
the representation theories of all
algebras over $K$.
Accordingly, we may realistically hope to classify
the indecomposable finite-dimensional modules only
for the tame algebras.
Among the tame algebras are  the
\emph{representation-finite algebras}, having
only finitely many isomorphism classes of indecomposable
modules, for which the representation theory is usually
well understood.
A prominent class of tame algebras is formed
by the \emph{algebras of polynomial growth} for which the number
of one-parameter families
of indecomposable modules in each dimension $d$ is bounded by $d^m$,
for some positive integer $m$ (depending only on the algebra) \cite{Sk0}.
The representation theory of tame self-injective algebras
of polynomial growth is currently well understood
(see \cite{BES2}, \cite{Sk1}, \cite{Sk2} for some general results).
On the other hand, for tame algebras which are not of polynomial growth, there
are few general tools, and  it is open in general how to describe the basic algebras
of tame self-injective algebras of non-polynomial growth.

As a step in this direction, we study in this paper
a wide class of tame symmetric algebras related
to tame blocks of group algebras, 
and to algebras constructed from 
surface triangulations. These 
have connections with group theory, topology, singularity theory and cluster
algebras, for a discussion and references, see the introduction of \cite{ESk4}. 
In the 
modular representation theory of finite groups, 
 representation-infinite tame blocks occur only
in characteristic 2,  and their defect groups are dihedral,
semidihedral, or  quaternion 2-groups (see \cite{BoD}, \cite{E5}).
In order to study such blocks,  algebras
of dihedral, semidihedral, and quaternion type have been introduced and
investigated, they are defined  over algebraically closed fields of arbitrary
characteristic (see \cite{E5}).

Recently, cluster theory has led to new directions. 
Inspired by this we study in \cite{ESk4} a class of symmetric algebras
which are defined in terms of surface triangulations, which we call {\it weighted surface algebras}. They
are tame and we show that they are (with one exception) 
periodic as algebras, of period $4$.
We observe that most algebras of quaternion type occur in this setting.
Furthermore,  most algebras of dihedral type, and of semidihedral type 
occur naturally as degenerations of these weighted surface algebras. 
This suggests that blocks of tame representation
type are part of a much wider context, which also 
connects with other parts of mathematics. 
This paper is part of a programme, exploring such a context.

As for blocks, it is natural to divide this programme into three parts.
For tame blocks, those  
with dihedral defect groups have 
been studied extensively
(see 
\cite{Br},
\cite{Do},
\cite{DF},
\cite{E1},
\cite{E2},
\cite{E3},
\cite{E4},
\cite{E5},
\cite{EM},
\cite{ESk1},
\cite{L},
\cite{Ri}
for some results on the structure of these 
blocks and their representations).
In this paper, we investigate 
the general setting  which 
includes such blocks.
We will introduce algebras of generalized dihedral type, and
we  provide a complete classification.
Before giving our definition, we recall basic definitions and concepts.

\bigskip

Throughout, $K$ will denote a fixed algebraically closed field.
By an algebra we mean an associative, finite-dimensional $K$-algebra
with an identity.
For an algebra $A$, we denote by $\mod A$ the category of
finite-dimensional right $A$-modules.
An algebra $A$ is called \emph{self-injective}
if $A_A$ is an injective module, or equivalently,
the projective modules in $\mod A$ are injective.
Two self-injective algebras $A$ and $\Lambda$ are said
to be \emph{socle equivalent} if the quotient algebras
$A/\soc(A)$ and $\Lambda/\soc(\Lambda)$ are isomorphic.
Two socle equivalent self-injective algebras have the same 
non-projective indecomposable modules.
Amongst self-injective algebras, there is the important  class 
of  \emph{symmetric algebras} $A$ for which there exists
an associative, non-degenerate, symmetric, $K$-bilinear form
$(-,-): A \times A \to K$.
Classical examples of symmetric algebras are  
blocks of group algebras of finite groups and
the Hecke algebras of finite Coxeter groups.
If $A$ is symmetric and $e$ is an idempotent of $A$ then also $eAe$ is symmetric. 
We call
$eAe$ an idempotent algebra of $A$.

\bigskip

Let $A$ be an algebra.
Given a module $M$ in  $\mod A$, its \emph{syzygy}
is defined to be the kernel $\Omega_A(M)$ of a minimal
projective cover of $M$ in $\mod A$.
The syzygy operator $\Omega_A$ is an important tool
to construct modules in $\mod A$ and relate them.
For $A$ self-injective, it induces an equivalence
of the stable module category $\umod A$,
and its inverse is the shift of a triangulated structure
on $\umod A$ \cite{Ha1}.
A module $M$ in $\mod A$ is said to be \emph{periodic}
if $\Omega_A^n(M) \cong M$ for some $n \geq 1$, and if so
the minimal such $n$ is called the \emph{period} of $M$.

An important combinatorial and homological invariant
of the module category $\mod A$ of an algebra $A$
is its Auslander-Reiten quiver $\Gamma_A$.
Recall that $\Gamma_A$ is the translation quiver whose
vertices are the isomorphism classes of indecomposable
modules in $\mod A$, the arrows correspond
to irreducible homomorphisms, and the translation
is the Auslander-Reiten translation $\tau_A = D \Tr$.
For $A$ self-injective, we denote by $\Gamma_A^s$
the stable Auslander-Reiten quiver of $A$, obtained
from $\Gamma_A$ by removing the isomorphism classes
of projective modules and the arrows attached to them.
A stable tube is a translation quiver $\Gamma$
of the form $\mathbb{Z} \mathbb{A}_{\infty}/(\tau^r)$,
for some $r \geq 1$, and call $r$ the rank of $\Gamma$.
We note that, for a symmetric algebra $A$, we have
$\tau_A = \Omega_A^2$ (see \cite[Corollary~IV.8.6]{SY}).
Hence  the indecomposable modules in stable tubes are precisely
the indecomposable periodic modules.
\bigskip

Let $A$ be an algebra. We say that $A$ is of 
\emph{generalized dihedral type}
if it satisfies the following conditions:
\begin{enumerate}
 \item[(1)]
  $A$ is symmetric, indecomposable, and tame, with the Grothendieck group $K_0(A)$
of rank at least $2$.
 \item[(2)]
  The stable Auslander-Reiten quiver $\Gamma_A^s$ of $A$ consists
  of the following components:
  \begin{enumerate}
   \item[(i)]
    stable tubes of ranks $1$ and $3$;
   \item[(ii)]
    non-periodic components of the form
    $\bZ \bA_{\infty}^{\infty}$
   or $\bZ \widetilde{\bA}_n$.
   We assume that there is at least one such component.   
\end{enumerate}
 \item[(3)]
  $\Omega_A$ fixes all stable tubes
  of rank $3$ in $\Gamma_A^s$.
\end{enumerate}
We note that every algebra of generalized dihedral type
is representation-infinite.

This is an algebraic definition, in terms of homological properties, 
and it is invariant under
Morita equivalence.
Blocks  with dihedral defect groups (with at least two simple modules) are examples. 
Geometrically, and motivated by the results of \cite{ESk4} and \cite{ESk7},
we introduce 
biserial weighted surface algebras, they are associated to
triangulated surfaces with arbitrarily oriented triangles. 
As well we 
study  distinguished idempotent algebras of these algebras
with respect to families of 2-triangle disks (see Section~\ref{sec:idempalg}).
With this, the algebraic version and the geometric version are the
same,  which is the main result of this paper.

\begin{main-theorem}
\label{th:main1}
Let $A$ be a basic algebra over an algebraically
closed field $K$.
Then the following statements are equivalent:
\begin{enumerate}
  \item[(i)]
    $A$ is of generalized dihedral type.
  \item[(ii)]
    $A$ is socle equivalent to 
    an idempotent algebra
    $B(S,\vv{T},\Sigma,m_{\bullet})$ of 
    a biserial weighted
    surface algebra $B(S,\vv{T},m_{\bullet})$,
    with respect to a collection $\Sigma$ of $2$-triangle disks 
    of $\vv{T}$.
\end{enumerate}
Moreover, if $K$ is of characteristic different from $2$, we
may replace in (ii) ``socle equivalent'' by ``isomorphic''.
\end{main-theorem}

The geometric version, that is, part (ii), provides a complete
description of these algebras by quivers and relations
(see Sections \ref{sec:bisweight}, \ref{sec:socdef}, \ref{sec:idempalg}).
The algebras in this theorem are in particular
socle deformations of the wider class of Brauer graph algebras 
which have attracted a lot of attention.

In fact, any Brauer graph algebra occurs as an idempotent algebra 
of some biserial weighted surface algebra, see \cite{ESk6}.
Condition (3) in the definition of generalized dihedral
type occurs naturally in the geometric version, it 
comes from  degenerating  relations around
a triangle for a weighted surface
algebra. This condition was  not part of the original definition
of algebras of dihedral type in \cite{E3}, but it holds for blocks
with dihedral defect groups. 

The main restriction for blocks is that the Cartan matrix of
a block must be non-singular, and that there are  few simple modules
and tubes of rank 3. This motivates the following definition:
The algebra $A$ is of strict dihedral type if it satisfies
the above conditions (1), (2) and (3), 
the Cartan matrix of $A$ is non-singular and in addition
the number $\ell(A)$ of simple modules is at most $3$, 
 and $\Gamma_A^s$ has $\ell(A)-1$  stable tubes of rank $3$. These include
all non-local blocks of group algebras with dihedral defect groups.

We characterize these 
algebras 
among all algebras of generalized dihedral
type.

\begin{main-theorem}
\label{th:main2}
Let $A$ be an  algebra over an algebraically
closed field $K$.
Then the following statements are equivalent:
\begin{enumerate}[(i)]
  \item
    $A$ is of strict dihedral type.
  \item
    $A$ is of generalized dihedral type
    and the Cartan matrix $C_A$ of $A$
    is non-singular.
\end{enumerate}
\end{main-theorem}

Moreover, we have the following consequences of  
Theorems \ref{th:main1} and \ref{th:main2}.

\begin{corollary*}
\label{cor:main3}
Let $A$ be an algebra of generalized dihedral type 
over an algebraically closed field $K$.
Then $A$ is a biserial algebra.
\end{corollary*}

\begin{corollary*}
\label{cor:main4}
Let $A$ be an algebra of generalized dihedral type 
over an algebraically closed field $K$, 
with the Grothendieck group $K_0(A)$
of rank at least $4$.
Then the Cartan matrix $C_{A}$ of $A$ is singular.
\end{corollary*}

We note that every algebra of generalized dihedral type
admits at least one non-periodic simple module
(see Corollary~\ref{cor:3.2}).
In the course of the proof of Theorem~\ref{th:main1}
we will also establish  the following corollary.

\begin{corollary*}
\label{cor:main5}
Let $A$ be an algebra of generalized dihedral type 
over an algebraically closed field $K$.
Then the following statements are equivalent:
\begin{enumerate}[(i)]
  \item
    All simple modules in $\mod A$ are non-periodic.
  \item
    The Gabriel quiver $Q_A$ of $A$ is $2$-regular.
\end{enumerate}
\end{corollary*}

There are only very few algebras
of generalized dihedral type whose stable 
Auslander-Reiten quiver admits a non-periodic
component of Euclidean type $\bZ \widetilde{\bA}_n$, they are
described in 
Corollary~\ref{cor:8.6}.

In fact, we have the following consequences
of Theorem~\ref{th:main1}
and the results established in 
\cite[Sections 2~and~4]{ESk1}.

\begin{corollary*}
\label{cor:main6}
Let $A$ be an algebra of generalized dihedral type 
over an algebraically closed field $K$.
Then the following statements are equivalent:
\begin{enumerate}[(i)]
  \item
    $A$ is of polynomial growth.
  \item
    $\Gamma_A^s$ admits a component of Euclidean
    type $\widetilde{\bA}_n$.
  \item
    $A$ is of strict dihedral type
    and $\Gamma_A^s$ consists of one component
    of Euclidean type $\widetilde{\bA}_{1,3}$
    or $\widetilde{\bA}_{3,3}$, at most two
    stable tubes of rank $3$, and infinitely many
    stable tubes of rank $1$.
\end{enumerate}
\end{corollary*}

For details on Euclidean components of types 
$\widetilde{A}_{n}$ and $\widetilde{A}_{p, q}$ we refer to 
\cite{ESk1}.

\begin{corollary*}
\label{cor:main7}
Let $A$ be an algebra of generalized dihedral type 
over an algebraically closed field $K$.
Then the following statements are equivalent:
\begin{enumerate}[(i)]
  \item
    $A$ is a tame algebra of non-polynomial growth.
  \item
    $\Gamma_A^s$ admits a component of the form
    $\bZ \bA_{\infty}^{\infty}$.
  \item
    $\Gamma_A^s$ consists of 
    a finite number of stable tubes of rank $3$, 
    infinitely many stable tubes of rank $1$,
    and infinitely many components of the form
    $\bZ \bA_{\infty}^{\infty}$.
\end{enumerate}
\end{corollary*}

There are only finitely many
stable tubes of rank $3$ in the stable
Auslander-Reiten quivers of biserial weighted surface algebras, 
explicitly we have the following corollary
(see the end of Section~\ref{sec:bisweight}).

\begin{corollary*}
\label{cor:main8}
Let $A$ be a basic self-injective algebra which is socle equivalent to
a biserial weighted surface algebra $B(S,\vv{T},m_{\bullet})$
of a directed triangulated surface $(S,\vv{T})$.
Then the number of stable tubes of rank $3$
in the stable Auslander-Reiten quiver $\Gamma_A^s$ 
of $A$ is equal to the number of triangles
in the triangulation $T$  of the surface $S$.
\end{corollary*}

This paper is organized as follows.
In Section~\ref{sec:bisalg}
we recall the structure
of stable Auslander-Reiten quivers of representation-infinite
self-injective special biserial algebras.
Section~\ref{sec:periodicalg}
contains  a complete description of representation-infinite
tame symmetric algebras for which all simple modules
are periodic of period $3$.
In Section~\ref{sec:triangulation}
we recall from  \cite{ESk4} the definition of 
 triangulation quivers and show that they
arise naturally from orientation of triangles of triangulated
surfaces.
In Section~\ref{sec:bisweight} 
we define biserial weighted triangulation algebras 
and describe their basic properties.
Section~\ref{sec:socdef} 
is devoted to socle deformations of biserial weighted
triangulation algebras and their basic properties. 
In Section~\ref{sec:idempalg}
we introduce the idempotent algebras of biserial weighted
surface algebras occurring in Theorem~\ref{th:main1}.
Section~\ref{sec:strictdih}
recalls known results on algebras of strict dihedral type
and exhibits presentations of these algebras 
as biserial weighted surface algebras
and their socle deformations.
In Sections \ref{sec:proof1} and \ref{sec:proof2} 
we prove  
Theorems \ref{th:main1} and \ref{th:main2}.

For general background on the relevant representation theory
we refer to the books 
\cite{ASS},
\cite{E5},
\cite{SS},
\cite{SY}.

\section{Special biserial algebras}\label{sec:bisalg}

A \emph{quiver} is a quadruple $Q = (Q_0, Q_1, s, t)$
consisting of a finite set $Q_0$ of vertices,
a finite set $Q_1$ of arrows,
and two maps $s,t : Q_1 \to Q_0$ which associate
to each arrow $\alpha \in Q_1$ its source $s(\alpha) \in Q_0$
and  its target $t(\alpha) \in Q_0$.
We denote by $K Q$ the path algebra of $Q$ over $K$
whose underlying $K$-vector space has as its basis
the set of all paths in $Q$ of length $\geq 0$, and
by $R_Q$ the arrow ideal of $K Q$ generated by all paths $Q$
of length $\geq 1$.
An ideal $I$ in $K Q$ is said to be \emph{admissible}
if there exists some $m \geq 2$ such that
$R_Q^m \subseteq I \subseteq R_Q^2$.
If $I$ is an admissible ideal in $K Q$, then
the quotient algebra $K Q/I$ is called
a bound quiver algebra, and is a finite-dimensional
basic $K$-algebra.
Moreover, $K Q/I$ is indecomposable if and only if
$Q$ is connected.
Every basic, indecomposable, finite-dimensional
$K$-algebra $A$ has a bound quiver presentation
$A \cong K Q/I$, where $Q = Q_A$ is the \emph{Gabriel
quiver} of $A$ and $I$ is an admissible ideal in $K Q$
(see \cite[Section~II.3]{ASS}).
For a bound quiver algebra $A = KQ/I$, we denote by $e_i$,
$i \in Q_0$, the associated complete set of pairwise
orthogonal primitive idempotents of $A$, and by
$S_i = e_i A/e_i \rad A$ (respectively, $P_i = e_i A$),
$i \in Q_0$, the associated complete family of pairwise
non-isomorphic simple modules (respectively, indecomposable
projective modules) in $\mod A$.

Following \cite{SW}, an algebra $A$ is said to be
\emph{special biserial} if $A$ is isomorphic
to a bound quiver algebra $K Q/I$, where the bound
quiver $(Q,I)$ satisfies the following conditions:
\begin{enumerate}[(a)]
 \item[(a)]
  each vertex of $Q$ is a source and target of at most two arrows,
 \item[(b)]
  for any arrow $\alpha$ in $Q$ there are at most
  one arrow $\beta$ and at most one arrow $\gamma$
  with $\alpha \beta \notin I$ and $\gamma \alpha \notin I$.
\end{enumerate}
Moreover, if $I$ is generated by paths of $Q$, then
$A = K Q/I$ is said to be a \emph{string algebra} \cite{BR}.
Every special biserial algebra is a biserial algebra \cite{Fu},
that is, the radical of any indecomposable non-uniserial projective, 
left or right,
module is a sum of two uniserial modules whose intersection
is simple or zero.
Moreover, every representation-finite biserial algebra
is special biserial, by \cite[Lemma~2]{SW}.
On the other hand, there are many biserial algebras
which are not special biserial.
In fact, it follows from the main
result of \cite{PS} that the class
of special biserial algebras coincides with the class
of biserial algebras 
which admit simply connected
Galois coverings.
It has been proved in \cite[Theorem~1.4]{WW}
that every special biserial algebra is a quotient algebra
of a symmetric special biserial algebra.
We also mention that, if $A$ is a self-injective
special biserial algebra, then $A/\soc(A)$ is a  string algebra.

The following fact has been proved by Wald and Waschb\"usch
in \cite{WW} (see also \cite{BR} and \cite{DS3} for alternative
proofs).

\begin{proposition}
\label{prop:2.1}
Every special biserial algebra is tame.
\end{proposition}

We refer to \cite{CB2} and \cite{VFCB} for the structure and
tameness of arbitrary biserial algebras.

The following two theorems from
\cite[Theorems~2.1 and 2.2]{ESk1}
describe the structure of the stable Auslander-Reiten quivers
of representation-infinite self-injective special biserial algebras.

\begin{theorem}
\label{th:2.2}
Let $A$ be a special biserial self-injective algebra.
Then the following statements are equivalent:
\begin{enumerate}[(i)]
  \item
    $A$ is representation-infinite of polynomial growth.
  \item
    $\Gamma_A^s$ admits a component of Euclidean
    type $\widetilde{\bA}_{p,q}$.
  \item
    There are positive integers $m,p,q$ such that   
    $\Gamma_A^s$ is a disjoint union of
    $m$ components of the form $\bZ \widetilde{\bA}_{p,q}$,
    $m$ components of the form $\bZ \bA_{\infty}/(\tau^p)$,
    $m$ components of the form $\bZ \bA_{\infty}/(\tau^q)$,
    and infinitely many components of the form 
    $\bZ \bA_{\infty}/(\tau)$.
\end{enumerate}
\end{theorem}

\begin{theorem}
\label{th:2.3}
Let $A$ be a special biserial self-injective algebra.
Then the following statements are equivalent:
\begin{enumerate}[(i)]
  \item
    $A$ is of non-polynomial growth.
  \item
    $\Gamma_A^s$ admits a component of the form
    $\bZ \bA_{\infty}^{\infty}$.
  \item
    $\Gamma_A^s$ is a disjoint union of
    a finite number of components of the form
    $\bZ \bA_{\infty}/(\tau^n)$ with $n > 1$,
    infinitely many components of the form 
    $\bZ \bA_{\infty}^{\infty}$,
    and infinitely many components of the form 
    $\bZ \bA_{\infty}^{\infty}/(\tau)$.
\end{enumerate}
\end{theorem}

We have also the following fact from \cite[Proposition~2.7]{ESk6}.
\begin{proposition}
\label{prop:2.4}
Let $A$ be a symmetric special biserial algebra
and $e$ an idempotent of $A$. 
Then $e A e$ is a symmetric special biserial algebra.
\end{proposition}

\section{Tame periodic algebras of period three}\label{sec:periodicalg}

Let $A$ be an algebra and $A^e = A^{\op} \otimes_K A$
the associated enveloping algebra.
Then $\mod A^e$ is the category of finite-dimensional
$A$-$A$-bimodules.
An algebra $A$ is said to be \emph{periodic} if $\Omega_{A^e}^d(A) \cong A$
in $\mod A^e$ for some $d\geq 1$, and if so the minimal such $d$ 
is called the \emph{period} of $A$.
It is known that if $A$ is a periodic algebra of period $d$,
then $A$ is  self-injective  and 
every indecomposable non-projective module $M$ in $\mod A$
is periodic of period dividing $d$ (see \cite[Theorem~IV.11.19]{SY}).

Recall that the preprojective algebra $P(\bD_4)$ of Dynkin
type $\bD_4$ over $K$ is the bound quiver algebra
given by the quiver
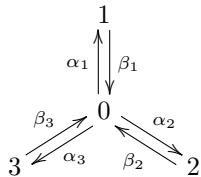
\[
  \xymatrix@C=1.9pc@R=.75pc{
      & 1  \ar@<+.5ex>[dd]^{\beta_1}
      \\ \\
      & 0  \ar@<+.5ex>[uu]^{\alpha_1}
           \ar@<+.5ex>[dr]^{\alpha_2}
           \ar@<+.5ex>[dl]^{\alpha_3}
           \\
      3  \ar@<+.5ex>[ru]^{\beta_3}
      &&  2  \ar@<+.5ex>[lu]^{\beta_2}
  }
\]
and the relations
$\alpha_1 \beta_1 +\alpha_2 \beta_2 +\alpha_3 \beta_3 = 0$,
$\beta_1\alpha_1 = 0$,
$\beta_2\alpha_2 = 0$,
$\beta_3\alpha_3 = 0$.

The following 
shows
that there is a unique  representation-infinite 
tame symmetric algebra such that  all simple modules
are periodic of period $3$.

\begin{theorem}
\label{th:3.1}
Let $A$ be a basic, indecomposable, tame symmetric algebra
of infinite representation type over an algebraically closed
field $K$.
Then the following statements are equivalent:
\begin{enumerate}[(i)]
  \item
    All simple modules in $\mod A$ are periodic of period $3$.
  \item
    $A$ is a periodic algebra of period $3$.
  \item
    $K$ has characteristic $2$ and $A$ is isomorphic 
    to the preprojective algebra $P(\bD_4)$.
\end{enumerate}
\end{theorem}

\begin{proof}
The implication (ii) $\Rightarrow$ (i) is obvious (see \cite[Theorem~IV.11.19]{SY}).
The implication (iii) $\Rightarrow$ (ii) follows from \cite[(2.10)]{ESn}).
Assume now that all simple modules in $\mod A$ are periodic of period $3$.
Then it follows from \cite[Theorem~1.2]{BES1}
and \cite[Theorem~3.7]{ESk2} that $A$ is socle equivalent
to the preprojective algebra $P(\mathbb{D}_4)$.
Since $A$ is symmetric, applying 
\cite[Proposition~6.8]{BES2}
and
\cite[Theorems~1 and Theorems~2]{BS},
we conclude that $K$ is of characteristic $2$ 
and $A$ is isomorphic to $P(\mathbb{D}_4)$.
Hence the implication (i) $\Rightarrow$ (ii) holds.
\end{proof}

Since the stable Auslander-Reiten quiver of an algebra
of generalized dihedral type has at least one non-periodic
component, we have the following consequence, which
will be
essential for the proof of Theorem~\ref{th:main1}.

\begin{corollary}
\label{cor:3.2}
Let $A$ be an algebra of generalized dihedral type.
Then $\mod A$ admits at least one non-periodic simple module.
\end{corollary}

\begin{proof}
By Theorem~\ref{th:3.1}, it is enough to show that every periodic simple module
in $\mod A$ has period $3$.
Suppose $S$ is a periodic module in $\mod A$.
Then $S$ belongs to a stable tube of $\Gamma_A^s$
of rank $1$ or $3$.
If $S$ belongs to a stable tube of rank $3$
then by the condition (3) $S$ is of period $3$.
On the other hand, if $S$ belongs to a stable tube of rank $1$
then $\Omega_A^2(S) \cong S$,
and then $A$ is a local algebra of finite representation type,
a contradiction.
\end{proof}

\begin{remark}
\label{rem:3.3}
Suppose $A$ is an algeba of generalized dihedral type and $S$ is a non-periodic
simple module in $\mod A$.
Then $S$ belongs to a component of the form $\bZ \bA_{\infty}^{\infty}$
or $\bZ \widetilde{\bA}_n$.
This means that if $P(S)$ the projective cover of $S$,
then $\rad P(S)/S$ is a direct sum of two indecomposable
modules (see \cite[Section~I.7]{E5}).
\end{remark}

\section{Triangulation quivers of surfaces}\label{sec:triangulation}

We  recall now the definition of   triangulation quivers associated to
directed triangulated surfaces, as introduced in \cite{ESk4}, and we 
present two 
examples.

In this  paper, by a \emph{surface}
we mean a connected, compact, $2$-dimensional real
manifold $S$, orientable or non-orientable,
with or without boundary.
It is well known that every surface $S$ admits
an additional structure of a finite
$2$-dimensional triangular cell complex,
and hence a triangulation, by the deep Triangulation Theorem
(see for example \cite[Section~2.3]{Ca}).

For a natural number $n$, we denote by $D^n$ the unit disk
in the $n$-dimensional Euclidean space $\bR^n$,
formed by all points of distance $\leq 1$ from the origin.
The boundary $\partial D^n$ of $D^n$ is the unit sphere
$S^{n-1}$ in $\bR^n$, the points of distance $1$
from the origin.
Further, by an $n$-cell we mean a topological space,
denoted by $e^n$,
homeomorphic to the open disk $\intt D^n = D^n \setminus \partial D^n$.
In particular,
$D^0$ and $e^0$ consist of a single point,
and $S^0 = \partial D^1$ consists of two points.

We refer to \cite[Appendix]{H} for some basic topological
facts about cell complexes.

Let $S$ be a surface.
In this  paper, by a 
\emph{finite $2$-dimensional triangular cell complex structure} 
on $S$ we mean a family of continous maps
$\varphi_i^n : D_i^n \to X$, with $n \in \{0,1,2\}$
and $D_i^n = D^n$,
for $i$ in a finite index set,
satisfying the following conditions:
\begin{enumerate}[(1)]
 \item
  Each $\varphi_i^n$ restricts to a homeomorphism from the interior\ 
  $\intt D_i^n$ to the $n$-cell $e_i^n = \varphi_i^n(\intt D_i^n)$ of $S$,
  and these cells are all disjoint and their union is $S$.
 \item
  For each $2$-dimensional cell $e_i^2$, the set 
  $\phi_i^2(\partial D_i^2)$
  is  the union of $k$ $1$-cells and $m$ $0$-cells,
  with $k \in \{2,3\}$ and $m \in \{1,2,3\}$.
\end{enumerate}
Then the closure  $\varphi_i^2(D_i^2)$ of a $2$-cell $e_i^2$
is called a \emph{triangle} of $S$,
and the closure  $\varphi_i^1(D_i^1)$ of a $1$-cell $e_i^1$
is called an \emph{edge} of $S$.
The collection $T$ of all triangles  $\varphi_i^2(D_i^2)$ 
is said to be a \emph{triangulation} of $S$.

We assume that such a triangulation $T$ of $S$
has at least two 
different edges,
or equivalently, there are at least two 
different
$1$-cells in the corresponding triangular cell complex structure on $S$.
Then $T$ is a finite collection of triangles
of the form
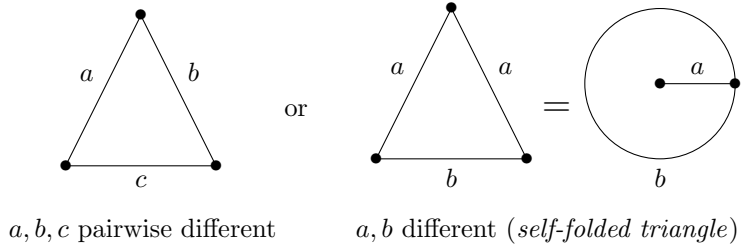
\begin{gather*}
\qquad
\begin{tikzpicture}[auto]
\coordinate (a) at (0,2);
\coordinate (b) at (-1,0);
\coordinate (c) at (1,0);
\draw (a) to node {$b$} (c)
(c) to node {$c$} (b);
\draw (b) to node {$a$} (a);
\node (a) at (0,2) {$\bullet$};
\node (b) at (-1,0) {$\bullet$};
\node (c) at (1,0) {$\bullet$};
\end{tikzpicture}
\qquad
\raisebox{7ex}{\mbox{or}}
\qquad
\begin{tikzpicture}[auto]
\coordinate (a) at (0,2);
\coordinate (b) at (-1,0);
\coordinate (c) at (1,0);
\draw (c) to node {$b$} (b)
(b) to node {$a$} (a);
\draw (a) to node {$a$} (c);
\node (a) at (0,2) {$\bullet$};
\node (b) at (-1,0) {$\bullet$};
\node (c) at (1,0) {$\bullet$};
\end{tikzpicture}
%
%
\raisebox{7ex}{\LARGE =}
\ \,
\begin{tikzpicture}[auto]
\coordinate (c) at (0,0);
\coordinate (a) at (1,0);
\coordinate (b) at (0,-1);
\draw (c) to node {$a$} (a);
\draw (b) arc (-90:270:1) node [below] {$b$};
\node (a) at (1,0) {$\bullet$};
\node (c) at (0,0) {$\bullet$};
\end{tikzpicture}
%
\\
\mbox{$a,b,c$ pairwise different}
\qquad
\quad
\mbox{$a,b$ different (\emph{self-folded triangle})}
\end{gather*}
such that every edge of such a triangle in $T$ is either
the edge of exactly two triangles, or is the self-folded
edge, or else lies on the boundary.
A  given surface $S$ admits many
finite $2$-dimensional triangular cell complex structures, and hence
triangulations.
We refer to \cite{Ca}, \cite{KC}, \cite{Ki} for
general background on surfaces and
constructions of surfaces from plane models.

Let $S$ be a surface and
$T$ a triangulation $S$.
To each triangle $\Delta$ in $T$ we may associate an orientation
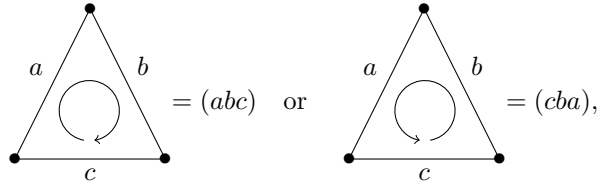
\[
\begin{tikzpicture}[auto]
\coordinate (a) at (0,2);
\coordinate (b) at (-1,0);
\coordinate (c) at (1,0);
\coordinate (d) at (-.08,.25);
\draw (a) to node {$b$} (c)
(c) to node {$c$} (b)
(b) to node {$a$} (a);
\draw[->] (d) arc (260:-80:.4);
\node (a) at (0,2) {$\bullet$};
\node (b) at (-1,0) {$\bullet$};
\node (c) at (1,0) {$\bullet$};
\end{tikzpicture}
\raisebox{7ex}{\!\!$=(abc)$}
\raisebox{7ex}{\quad or \ \ }
\begin{tikzpicture}[auto]
\coordinate (a) at (0,2);
\coordinate (b) at (-1,0);
\coordinate (c) at (1,0);
\coordinate (d) at (.08,.25);
\draw (a) to node {$b$} (c)
(c) to node {$c$} (b)
(b) to node {$a$} (a);
\draw[->] (d) arc (-80:260:.4);
\node (a) at (0,2) {$\bullet$};
\node (b) at (-1,0) {$\bullet$};
\node (c) at (1,0) {$\bullet$};
\end{tikzpicture}
\raisebox{7ex}{\!\!$=(cba)$,}
\]
if $\Delta$ has pairwise different edges $a,b,c$, and
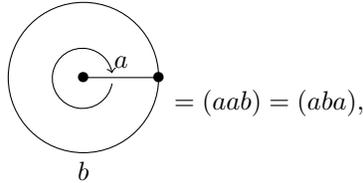
\[
\begin{tikzpicture}[auto]
\coordinate (c) at (0,0);
\coordinate (a) at (1,0);
\coordinate (b) at (0,-1);
\coordinate (d) at (.38,-.08);
\draw (c) to node {$a$} (a);
\draw (b) arc (-90:270:1) node [below] {$b$};
\node (a) at (1,0) {$\bullet$};
\node (c) at (0,0) {$\bullet$};
\draw[->] (d) arc (-10:-350:.4);
\end{tikzpicture}
\raisebox{7ex}{$=(aab)=(aba)$,}
\]
if $\Delta$ is self-folded, with the self-folded edge $a$,
and the other edge $b$.
Fix an orientation of each triangle $\Delta$ of $T$,
and denote this choice by $\vv{T}$.
Then
the pair $(S,\vv{T})$ is said to be a
\emph{directed triangulated surface}.
To each directed triangulated surface $(S,\vv{T})$
we associate the quiver $Q(S,\vv{T})$ whose vertices
are the edges of $T$ and the arrows are defined as
follows:
\begin{enumerate}[(1)]
 \item
  for any oriented triangle $\Delta = (a b c)$ in $\vv{T}$
  with pairwise different edges $a,b,c$, we have the cycle
  
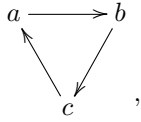
\[
    \xymatrix@C=.8pc{a \ar[rr] && b \ar[ld] \\ & c \ar[lu]}
    \raisebox{-7ex}{,}
  \]
 \item
  for any self-folded triangle $\Delta = (a a b)$ in $\vv{T}$,
  we have the quiver
  
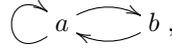
\[
    \xymatrix{ a \ar@(dl,ul)[] \ar@/^1.5ex/[r] & b \ar@/^1.5ex/[l]} ,
  \]
 \item
  for any boundary edge $a$ in ${T}$,
  we have the loop
  
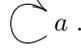
\[
    {\xymatrix{ a \ar@(dl,ul)[]}} .
  \]
\end{enumerate}
Then $Q = Q(S,\vv{T})$ is a triangulation quiver in
the following sense (introduced independently by Ladkani
in \cite{La}).
For the history see the Acknowledgements in \cite{ESk4}.

A \emph{triangulation quiver} is a pair $(Q,f)$,
where $Q = (Q_0,Q_1,s,t)$ is a finite connected quiver
and $f : Q_1 \to Q_1$ is a permutation
on the set $Q_1$ of arrows of $Q$ satisfying
the following conditions:
\begin{enumerate}[(a)]
 \item[(a)]
  every vertex $i \in Q_0$ is the source and target of exactly two
  arrows in $Q_1$,
 \item[(b)]
  for each arrow $\alpha \in Q_1$, we have $s(f(\alpha)) = t(\alpha)$,
 \item[(c)]
  $f^3$ is the identity on $Q_1$.
\end{enumerate}

For the quiver $Q = Q(S,\vv{T})$ of a
directed triangulated surface $(S,\vv{T})$,
the pair $(Q,f)$ is a triangulation quiver, where
the permutation $f$ on its set of arrows
is defined as follows:
\begin{enumerate}[(1)]
 \item
  \raisebox{3ex}%
  {\xymatrix@C=.8pc{a \ar[rr]^{\alpha} && b \ar[ld]^{\beta} \\ & c \ar[lu]^{\gamma}}}%
  \quad
  $f(\alpha) = \beta$,
  $f(\beta) = \gamma$,
  $f(\gamma) = \alpha$,

  for an oriented triangle $\Delta = (a b c)$ in $\vv{T}$,
  with pairwise different edges $a,b,c$,
 \item
  \raisebox{0ex}%
  {\xymatrix{ a \ar@(dl,ul)[]^{\alpha} \ar@/^1.5ex/[r]^{\beta} & b \ar@/^1.5ex/[l]^{\gamma}}}
  \quad
  $f(\alpha) = \beta$,
  $f(\beta) = \gamma$,
  $f(\gamma) = \alpha$,

  for a self-folded triangle $\Delta = (a a b)$ in $\vv{T}$,\vspace{1mm}
 \item
  \raisebox{0ex}%
  {\xymatrix{ a \ar@(dl,ul)[]^{\alpha}}}
  \quad
  $f(\alpha) = \alpha$,\vspace{1mm}

  for a boundary edge $a$ of ${T}$.
\end{enumerate}

We note that, 
if $(Q, f)$ is a 
triangulation quiver, then
$Q$ is a $2$-regular quiver.
\emph{We will  consider only the trangulation
quivers with at least two vertices.}

We would like to mention that different directed triangulated
surfaces (even of different genus) may lead to the same
triangulation quiver (see \cite[Example~4.3]{ESk4}).

Let $(Q,f)$ be a triangulation quiver.
Then we have the involution $\bar{}: Q_1 \to Q_1$
which assigns to an arrow $\alpha \in Q_1$
the arrow $\bar{\alpha}$ with $s(\alpha) = s(\bar{\alpha})$
and $\alpha \neq \bar{\alpha}$.
Then we obtain another permutation $g: Q_1 \to Q_1$
of the set $Q_1$ of arrows of $Q$ such that
$g(\alpha) = \overbar{f(\alpha)}$ for any $\alpha \in Q_1$.
We denote by $\cO(g)$ the set of all $g$-orbits
in $Q_1$.

The following theorem and  its consequence
have been established in \cite[Section~4]{ESk4}
(see also Example~\ref{ex:8.2} for the case with two vertices).

\begin{theorem}
\label{th:4.1}
Let $(Q,f)$ be a triangulation quiver.
Then there exists a directed triangulated surface
$(S,\vv{T})$ such that $(Q,f) = (Q(S,\vv{T}),f)$.
\end{theorem}

\begin{corollary}
\label{cor:4.2}
Let $(Q,f)$ be a triangulation quiver.
Then $Q$ contains a loop $\alpha$ with $f(\alpha) = \alpha$ 
is and only if 
$(Q,f) = (Q(S,\vv{T}),f)$
for a directed triangulated surface $(S,\vv{T})$ with $S$
having non-empty boundary.
\end{corollary}

We will present now two  examples of triangulation
quivers associated to directed triangulated surfaces.
Further examples may be found in \cite{ESk4}.

\begin{example}
\label{ex:4.3} 
Let $S$ consist of one triangle $T$
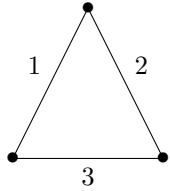
\[
\begin{tikzpicture}[auto]
\coordinate (a) at (0,2);
\coordinate (b) at (-1,0);
\coordinate (c) at (1,0);
\draw (a) to node {2} (c)
(c) to node {3} (b);
\draw (b) to node {1} (a);
\node (a) at (0,2) {$\bullet$};
\node (b) at (-1,0) {$\bullet$};
\node (c) at (1,0) {$\bullet$};
\end{tikzpicture}
\]
with the three pairwise different edges,
forming the boundary of $S$, and consider the orientation
$\vv{T}$ of $T$
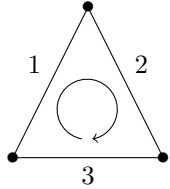
\[
\begin{tikzpicture}[auto]
\coordinate (a) at (0,2);
\coordinate (b) at (-1,0);
\coordinate (c) at (1,0);
\coordinate (d) at (-.08,.25);
\draw (a) to node {2} (c)
(c) to node {3} (b)
(b) to node {1} (a);
\draw[->] (d) arc (260:-80:.4);
\node (a) at (0,2) {$\bullet$};
\node (b) at (-1,0) {$\bullet$};
\node (c) at (1,0) {$\bullet$};
\end{tikzpicture}
\]
of $T$.
Then the triangulation quiver
$(Q(S,\vv{T}),f)$ is the quiver
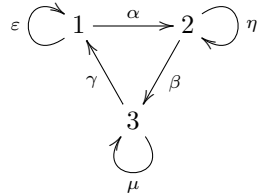
\[
  \xymatrix@C=.8pc{
     1 \ar@(dl,ul)[]^{\varepsilon} \ar[rr]^{\alpha} &&
     2 \ar@(ur,dr)[]^{\eta} \ar[ld]^{\beta} \\
     & 3 \ar@(dr,dl)[]^{\mu} \ar[lu]^{\gamma}}
\]
with $f$-orbits
$(\alpha\, \beta\, \gamma)$,
$(\varepsilon)$,
$(\eta)$,
$(\mu)$.
Observe that we have only one $g$-orbit
$(\alpha\, \eta\, \beta\, \mu\, \gamma\, \varepsilon)$
of arrows in $Q(S,\vv{T})$.
In particular, $|\cO(g)| = 1$.
\end{example}

\begin{example}
\label{ex:4.4} 
Let $S$  be the sphere with the triangulation $T$
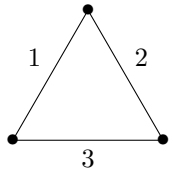
\[
\begin{tikzpicture}[auto]
\coordinate (a) at (0,1.73);
\coordinate (b) at (-1,0);
\coordinate (c) at (1,0);
\draw (a) to node {2} (c)
(c) to node {3} (b);
\draw (b) to node {1} (a);
\node (a) at (0,1.73) {$\bullet$};
\node (b) at (-1,0) {$\bullet$};
\node (c) at (1,0) {$\bullet$};
\end{tikzpicture}
\]
given by two unfolded triangles,
and $\vv{T}$ the following orientation
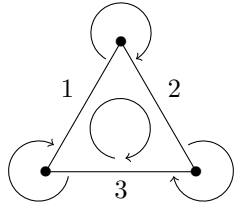
\[
\begin{tikzpicture}[auto]
\coordinate (a) at (0,1.73);
\coordinate (b) at (-1,0);
\coordinate (c) at (1,0);
\coordinate (d) at (-.08,.18);
\coordinate (aa) at (-.2,1.5);
\coordinate (bb) at (-.7,-.06);
\coordinate (cc) at (.7,-.06);
\draw (a) to node {2} (c)
(c) to node {3} (b)
(b) to node {1} (a);
\draw[->] (d) arc (260:-80:.4);
\node (a) at (0,1.73) {$\bullet$};
\node (b) at (-1,0) {$\bullet$};
\node (c) at (1,0) {$\bullet$};
\draw[->] (aa) arc (240:-60:.4);
\draw[->] (bb) arc (350:60:.4);
\draw[<-] (cc) arc (-170:120:.4);
\end{tikzpicture}
\]
of triangles of $T$.
Then the triangulation quiver
$(Q(S,\vv{T}),f)$ is the quiver 
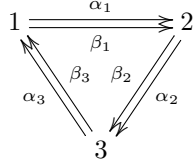
\[
  \xymatrix@R=3.pc@C=1.8pc{
    1
    \ar@<.35ex>[rr]^{\alpha_1}
    \ar@<-.35ex>[rr]_{\beta_1}
    && 2
    \ar@<.35ex>[ld]^{\alpha_2}
    \ar@<-.35ex>[ld]_{\beta_2}
    \\
    & 3
    \ar@<.35ex>[lu]^{\alpha_3}
    \ar@<-.35ex>[lu]_{\beta_3}
  }
\]
with the $f$-orbits 
$(\alpha_1 \, \alpha_2 \, \alpha_3)$ and $(\beta_1 \, \beta_2 \, \beta_3)$.
Then $\cO(g)$ consists of one $g$-orbit 
$(\alpha_1 \, \beta_2 \, \alpha_3 \, \beta_1 \, \alpha_2 \, \beta_3)$.
This triangulation quiver is called the
\emph{Markov quiver}
(see \cite{Lam}, \cite{M} for justification of this name).
\end{example}

\section{Biserial weighted triangulation algebras}\label{sec:bisweight}

Let $(Q,f)$ be a triangulation quiver, so that 
we have two permutations
$f : Q_1 \to Q_1$
and
$g : Q_1 \to Q_1$
on the set $Q_1$ of arrows of $Q$ such that $f^3$
is the identity on $Q_1$ and $g = \bar{f}$,
where $\bar{ } : Q_1 \to Q_1$
is the involution which assigns to an arrow
$\alpha \in Q_1$ the arrow $\bar{\alpha}$
with $s({\alpha}) = s(\bar{\alpha})$
and ${\alpha} \neq \bar{\alpha}$.
For each arrow $\alpha \in Q_1$, we denote by
$\cO(\alpha)$ the $g$-orbit of $\alpha$
in $Q_1$, and set
$n_{\alpha} = n_{\cO(\alpha)} = |\cO(\alpha)|$.
A function
\[
  m_{\bullet} : \cO(g) \to \bN^* = \bN \setminus \{0\}
\]
is said to be a \emph{weight function} of $(Q,f)$.
Write  $m_{\alpha} = m_{\cO(\alpha)}$ 
for $\alpha \in Q_1$.
For any arrow $\alpha \in Q_1$, we have 
the oriented cycle
\[
  B_{\alpha} = \Big( \alpha g(\alpha) \dots g^{n_{\alpha}-1}(\alpha)\Big)^{m_{\alpha}}
\]
of length $m_{\alpha} n_{\alpha}$.
The triple $(Q,f,m_{\bullet})$ is said to be a
\emph{weighted triangulation quiver}.

Let $(Q,f,m_{\bullet})$ be a weighted triangulation quiver.
We consider the quotient algebra
\[
  B(Q,f,m_{\bullet})
   = K Q / J(Q,f,m_{\bullet}),
\]
where $J(Q,f,m_{\bullet})$
is the ideal in the path algebra $KQ$ of $Q$ over $K$
generated by the elements:
\begin{enumerate}[(1)]
 \item
  $\alpha f(\alpha)$,
  for all arrows $\alpha \in Q_1$.
 \item
  $B_{\alpha} - B_{\bar{\alpha}}$,
  for all arrows $\alpha \in Q_1$,
\end{enumerate}
Then $B(Q,f,m_{\bullet})$ is said to be a
\emph{biserial  weighted triangulation algebra}.
Let $(S,\vv{T})$ be a directed triangulated surface,
$(Q(S, \vv{T}),f)$ the associated triangulation quiver,
and $m_{\bullet}$ a weight function of $(Q(S, \vv{T}),f)$.
Then the biserial weighted triangulation algebra $B(Q(S, \vv{T}),f,m_{\bullet})$
will be called a \emph{biserial weighted surface algebra},
and denoted by $B(S, \vv{T}, m_{\bullet})$.

\begin{remark} 
\label{rem:5.1}
We note that the Gabriel quiver of 
a biserial weighted triangulation algebra
$B(Q,f,m_{\bullet})$ 
is the subquiver of the triangulation quiver $(Q,f)$ obtained
by removing the loops $\alpha$ fixed by $g$ such that
$m_{\alpha}=1$. Namely, if $\alpha$ is such a loop 
then 
the element $B_{\alpha}$ occuring in the definition of 
$B(Q,f,m_{\bullet})$ 
is equal to $\alpha$ and therefore,
by condition (2) of that definition, $\alpha$ is in the square
of the radical. 
\end{remark}

The following proposition describes basic properties
of biserial weighted triangulation algebras.

\begin{proposition}
\label{prop:5.1}
Let $B = B(Q, f, m_{\bullet})$
be a biserial weighted triangulation algebra.
Then the following statements hold.
\begin{enumerate}[(i)]
 \item
  $B$ is  finite-dimensional special biserial  with 
  $\dim_K \! B = \sum_{\cO\in\cO(g)} m_{\cO} n_{\cO}^2$.
 \item
  $B$ is a tame symmetric algebra.
 \item
  $B$ is an algebra of generalized dihedral type.
\end{enumerate}
\end{proposition}

\begin{proof}
We set  $J = J(Q,f,m_{\bullet})$.
We recall that $Q$ has at least two vertices.

\smallskip

(i)
It follows from the definition 
and Remark~\ref{rem:5.1}
that $B$ is a finite-dimensional
special biserial algebra.
Let $i$ be a vertex of $Q$ and $\alpha$, $\bar{\alpha}$
the two arrows in $Q$ with source $i$.
Then the indecomposable projective right $B$-module
$P_i = e_i B$ at  vertex $i$ has  dimension
$\dim_K P_i = m_{\alpha} n_{\alpha} + m_{\bar{\alpha}} n_{\bar{\alpha}}$.
Indeed, $P_i$ has a basis given by $e_i$,
the cosets $u+J$ of all initial proper subwords $u$
of $B_{\alpha}$ and $B_{\bar{\alpha}}$,
and $B_{\alpha}+J = B_{\bar{\alpha}}+J$.
Hence we deduce that
\[
  \dim_K \! B = \sum_{\cO \in \cO(g)} m_{\cO} n_{\cO}^2 .
\]

(ii)
It is well known
(see \cite[Theorem~IV.2.2]{SY}) 
that $B$ is  symmetric  if and only if
there exists a $K$-linear form $\varphi : B \to K$
such that $\varphi(a b) = \varphi(b a)$ for all $a,b \in B$
and $\Ker \varphi$ does not contain any non-zero one-sided ideal
of $B$ (called a symmetrizing form).
We observe that, for any vertex $i$ of $Q$ and the arrows
$\alpha$, $\bar{\alpha}$ with source $i$, the indecomposable
projective module $P_i = e_i B$ has one-dimensional socle
generated by $B_{\alpha}+J = B_{\bar{\alpha}}+J$.
Clearly, we have also $\topp(P_i) = S_i = \soc(P_i)$.
We define a required symmetrizing form $\varphi : B \to K$
by assigning to the coset $u+J$ of a path $u$ in $Q$
the following element of $K$
\[
  \varphi(u+J) = \left\{
   \begin{array}{cl}
   1 & \mbox{if $u = B_{\alpha}$ for an arrow $\alpha \in Q_1$},\\
   0 & \mbox{otherwise},
   \end{array}
  \right.
\]
and extending to a $K$-linear form.
It follows from 
Proposition~\ref{prop:2.1} that $B$ is  tame.
\smallskip

(iii)
First we show that $B$ is representation-infinite.
Let $B' = B / \soc(B)$.
Then $B'$ is an algebra of the form $K Q/I$,
where $I$ is the ideal in $K Q$ generated by the
elements of the forms $\alpha f(\alpha)$ and $B_{\alpha}$,
for all arrows $\alpha \in Q_1$.
Moreover, $B$ is representation-infinite if and only if
$B'$ is representation-infinite.
We note that $K Q/I = K Q'/I'$, where $Q' = Q_B = Q_{B'}$,
with $Q_B$ the Gabriel quiver of $B$ 
and $Q_{B'}$ the Gabriel quiver of $B'$,
and $I'$ is the ideal in $K Q'$ generated by the elements
of the forms $\alpha f(\alpha)$, with $\alpha, f(\alpha)$ 
in $Q'_1$, and $B_{\gamma}$, for
all arrows $\gamma \in Q'_1$.
In particular, we conclude that
$B' = K Q' / I'$ is a string algebra.
It follows from 
\cite[Theorem~1]{SW} 
that $B'$ is representation-infinite if and only if
$(Q',I')$ admits a primitive walk.
For each arrow $\alpha \in Q_1$, 
we denote by $\alpha^{-1}$ the formal inverse of $\alpha$
and set 
$s(\alpha^{-1}) = t(\alpha)$
and
$t(\alpha^{-1}) = s(\alpha)$.
By a walk in $Q$ we mean a sequence
$w = \alpha_1 \dots \alpha_n$,
where $\alpha_i$ is an arrow or the inverse of an arrow in $Q$,
satisfying the conditions:
$t(\alpha_i) = s(\alpha_{i+1})$
and
$\alpha_{i+1} \neq \alpha_i^{-1}$
for any $i \in \{1,\dots,n-1\}$.
Moreover, 
$w$ is said to be a bipartite walk if,
for any $i \in \{1,\dots,n-1\}$,
exactly one of $\alpha_i$ and $\alpha_{i+1}$ is an arrow.
A walk $w = \alpha_1 \dots \alpha_n$ in $Q$ with
$s(\alpha_1) = t(\alpha_n)$ is called closed.
A closed walk $w$ in $Q$ is called a primitive walk
if $w \neq v^r$ for any closed walk $v$ in $Q$ 
and positive integer $r$.
We claim that for any arrow $\alpha \in Q_1$,
there is a bipartite primitive walk $w(\alpha)$ in $Q$ containing
the arrow $\alpha$.
Since $Q$ is a $2$-regular quiver,
we have two involutions
\,$\bar{ }: Q_1 \to Q_1$
and
${ }^*: Q_1 \to Q_1$.
The first involution assigns to each arrow
$\alpha \in Q_1$ the arrow $\bar{\alpha}$
with $s(\alpha) = s(\bar{\alpha})$
and $\alpha \neq \bar{\alpha}$.
The second involution assigns to each arrow
$\alpha \in Q_1$ the arrow ${\alpha}^*$
with $t(\alpha) = t({\alpha}^*)$
and $\alpha \neq {\alpha}^*$.
Consider the bijection
$h : Q_1 \to Q_1$ 
such that
$h(\alpha) = \overbar{\alpha^*}$
for any arrow $\alpha \in Q_1$.
Clearly, $h$ has finite order.
In particular,
for any arrow $\alpha \in Q_1$,
there exists a minimal positive integer $r$
such that
$h^r(\alpha) = \alpha$.
Then the required bipartite primitive walk
$w(\alpha)$ is of the form
\[
   \alpha (\alpha^*)^{-1} h(\alpha) \big(h(\alpha)^*\big)^{-1}
   \dots
   h^r (\alpha) \big(h^{r-1}(\alpha)^*\big)^{-1} 
   .
\]
Observe now that, 
if $\sigma$ is an arrow in $Q_1$ 
with $m_{\sigma} n_{\sigma} = 1$,
then $\sigma$ is a loop and $Q$ admits a subquiver
of the form
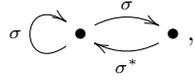
\[
 \xymatrix{ 
   \bullet \ar@(dl,ul)[]^{\sigma} \ar@/^1.5ex/[r]^{\bar{\sigma}} & 
   \bullet \ar@/^1.5ex/[l]^{\sigma^*}}
   ,
\]
with 
$f(\sigma) = \bar{\sigma}$,
$f(\bar{\sigma}) = \sigma^*$,
$f(\sigma^*) = {\sigma}$.
Moreover, for such a subquiver, the path $\sigma^*\bar{\sigma}$
is not in $I'$, because 
$\bar{\sigma} = g(\sigma^*)$ and $\sigma^* \neq g(\bar{\sigma})$.
Take an arrow $\alpha \in Q_1$ which is not a loop.
We denote by $w(\alpha)_0$ the primitive walk in $Q'$ 
obtained from $w(\alpha)$ by removing all loops $\sigma$
(respectively, the inverse loops $\sigma$)
such that $m_{\sigma} n_{\sigma} = 1$.
Then $w(\alpha)_0$ is a primitive walk in $(Q',I')$,
that is, $w(\alpha)_0$ does not contain a subpath
$v$ such that $v$ or $v^{-1}$ belongs to $I'$.
Therefore, $B'$ is representation-infinite,
and hence $B$ is representation-infinite.

Since $B$ is a representation-infinite special biserial
symmetric algebra, the structure of the stable 
Auslander-Reiten quiver $\Gamma_B^s$ of $B$ is described in 
Theorems \ref{th:2.2} and \ref{th:2.3}.
Hence, in order to prove that $B$ is an algebra of generalized 
dihedral type, 
it remains to show that,
if $\cT$ is a stable tube of rank $r \geq 2$ in $\Gamma_B^s$,
then $\cT$ is of rank $3$ and $\Omega_B$ fixes $\cT$.
By the general theory of special biserial algebras,
the stable tubes of ranks at least $2$ consist entirely
of string modules
(see \cite{BR}, \cite{DS3}, \cite{WW}).
Moreover, it follows from
\cite[Section~3]{BR} that the mouth of stable tubes
of ranks at least $2$ are formed by the uniserial 
string modules given by the arrows of the Gabriel
quiver $Q_B$ of $B$.
We recall that the Gabriel quiver $Q_B$ of $B$
is obtained from the quiver $Q$ by removing all loops $\alpha$
with $m_{\alpha} n_{\alpha} = 1$.
Further, if $\alpha$ is a loop in $Q_1$
with $m_{\alpha} n_{\alpha} = 1$,
then $P_{s(\alpha)} = e_{s(\alpha)} B$
is a uniserial module and $\rad P_{s(\alpha)} = \bar{\alpha} B$.
On the other hand, if $\alpha$ is an arrow in $Q_1$
with $m_{\alpha} n_{\alpha} \geq 2$ and $m_{ \bar{\alpha}} n_{ \bar{\alpha}} \geq 2$, 
then the indecomposable projective module
$P_{s(\alpha)} = e_{s(\alpha)} B$
is not uniserial, 
$\rad P_{s(\alpha)} = {\alpha} B + \bar{\alpha} B$,
and 
${\alpha} B \cap \bar{\alpha} B = \soc (P_{s(\alpha)})$
is the one-dimensional space generated by
$B_{\alpha} = B_{\bar{\alpha}}$.

Let $\alpha$ be an arrow in $Q_1$, and $U(\alpha) = \alpha B$.
Then $U(\alpha)$ is a uniserial module such that
$P_{s(\alpha)} / U(\alpha)$
is isomorphic to $U(f^{-1}(\alpha))$.
We also note that,
if $\alpha$ is a loop with $m_{\alpha} n_{\alpha} = 1$, 
then $U(\alpha) = \alpha B = B_{\alpha} B$
is simple and $P_{s(\alpha)} / U(\alpha) = U(\alpha^*)$
for the unique arrow $\alpha^*$ in $Q_1$
with $t(\alpha^*) = t(\alpha)$ and $\alpha^* \neq \alpha$.
We have the canonical short exact sequence in $\mod B$
\[
  0 \to
  U(f(\alpha)) \hookrightarrow
  P_{t(\alpha)} \xrightarrow{\pi_{\alpha}}
  U(\alpha) \to
  0
\]
with $\pi_{\alpha}$ being the projective cover,
and hence $\Omega_B(U(\alpha)) = U(f(\alpha))$.
In particular, we conclude that 
$\Omega_B(U(\alpha)) \cong U(\alpha)$
if $f(\alpha) = \alpha$, and $U(\alpha)$
is a periodic module of period $3$ if $f(\alpha) \neq \alpha$.
Since $B$ is a symmetric algebra we have
$\tau_B = \Omega_B^2$.
Hence, if $f(\alpha) = \alpha$,
the module $U(\alpha)$ forms the mouth of a stable tube
in $\Gamma_B^s$ of rank $1$.
On the other hand, if $f(\alpha) \neq \alpha$,
then the modules $U(\alpha)$, $U(f(\alpha))$, $U(f^2(\alpha))$ 
form the mouth of a stable tube
in $\Gamma_B^s$ of rank $3$,
and the Auslander-Reiten translation $\tau_B$ acts on these
modules as follows
\begin{align*}
  \tau_B U(\alpha) &= U\big(f^2(\alpha)\big),
 &
  \tau_B U\big(f^2(\alpha)\big) &= U\big(f(\alpha)\big),
 &
  \tau_B U\big(f(\alpha)\big) &= U(\alpha) .
\end{align*}
It follows from \cite[Section~3]{BR} that the uniserial modules
$U(\alpha)$, $\alpha \in Q_1$,
are the only string modules in $\mod B$
lying on the mouths of stable tubes in $\Gamma_B^s$.
Therefore, the stable tubes in $\Gamma_B^s$
are of ranks $1$ and $3$, and $\Omega_B$ fixes all
stable tubes of rank $3$ in $\Gamma_B^s$.
Summing up, 
we conclude that $B$ is an algebra of generalized dihedral type.
\end{proof}

\begin{proof}[Proof of Corollary~\ref{cor:main8}]
Let $A$ be a basic self-injective algebra
which is socle equivalent to a biserial weighted surface algebra
$B = B(S, \vv{T}, m_{\bullet})$ of a directed triangulated 
surface $(S, \vv{T})$.
Since the stable Auslander-Reiten quivers $\Gamma_A^s$
and $\Gamma_B^s$ are isomorphic, we may assume that $A = B$.
We observe now that there is a bijection between the triangles
in $T$ and the $f$-orbits of length $3$ in the associated 
triangulation quiver $(Q(S, \vv{T}),f)$, 
defined in Section~\ref{sec:triangulation}.
Further, it follows from the final part of the above proof
of Proposition~\ref{prop:5.1}, that the
mouth of stable tubes of rank $3$
in $\Gamma_B^s$ are formed by the uniserial modules
associated to the arrows of $f$-orbits of length $3$ in
$(Q(S, \vv{T}),f)$.
Therefore, the statement of Corollary~\ref{cor:main8} follows.
\end{proof}

\section{Socle deformed biserial weighted triangulation algebras}\label{sec:socdef}

In this section we introduce socle deformations of biserial 
weighted triangulation algebras occurring in the characterization 
of algebras of generalized dihedral type.

For a positive integer $d$, we denote by $\alg_d(K)$ the affine
variety of associative $K$-algebra structures with identity on
the affine space $K^d$.
The general linear group $\GL_d(K)$ acts on $\alg_d(K)$
by transport of the structures, and the $\GL_d(K)$-orbits in
$\alg_d(K)$ correspond to the isomorphism classes of $d$-dimensional
algebras (see \cite{Kr} for details). We identify a $d$-dimensional
algebra $A$ with the point of $\alg_d(K)$ corresponding to it.
For two $d$-dimensional algebras $A$ and $B$, we say that $B$
is a \emph{degeneration} of $A$ ($A$ is a \emph{deformation} of $B$)
if $B$ belongs to the closure of the $\GL_d(K)$-orbit
of $A$ in the Zariski topology of $\alg_d(K)$.

Geiss' Theorem \cite{Ge} says that if $A$ and $B$ are two
$d$-dimensional algebras, $A$ degenerates to $B$ and $B$ is a tame
algebra, then $A$ is also a tame algebra (see also \cite{CB2}).
We will apply this theorem in the following special situation.

\begin{proposition}
\label{prop:6.1}
Let $d$ be a positive integer, and $A(t)$, $t \in K$,
be an algebraic family in $\alg_d(K)$ such that $A(t) \cong A(1)$
for all $t \in K \setminus \{0\}$.
Then $A(1)$ degenerates to $A(0)$.
In particular, if $A(0)$ is tame, then $A(1)$ is tame.
\end{proposition}

A family of algebras $A(t)$, $t \in K$, in $\alg_d(K)$
is said to be \emph{algebraic} if the induced map
$A(-) : K \to \alg_d(K)$ is a regular map of affine varieties.

Let $(Q,f)$ be a triangulation quiver.
A vertex $i \in Q_0$ is said to be a \emph{border vertex}
of $(Q,f)$ if there is a loop $\alpha$ at $i$ with $f(\alpha) = \alpha$,
which we call a \emph{border loop}.
We denote by $\partial(Q,f)$
the set of all border vertices of $(Q,f)$,
and call it the \emph{border} of $(Q,f)$.
Observe that, if $(S,\vv{T})$
is a directed triangulated surface with
$(Q(S,\vv{T}),f) = (Q, f)$,
then the border vertices of $(Q,f)$
correspond bijectively to the boundary edges
of the triangulation $T$ of $S$.
Hence, the border $\partial(Q,f)$ of $(Q,f)$
is not empty if and only if the boundary
$\partial S$ of $S$ is not empty.
A function
\[
  b_{\bullet} : \partial(Q,f) \to K
\]
is said to be a \emph{border function} of $(Q,f)$.

Let $(Q,f)$ be a triangulation quiver, 
and assume
that the border $\partial(Q,f)$ of $(Q,f)$
is not empty.
Let
$m_{\bullet} : \cO(g) \to \bN^*$
be a weight function
and
$b_{\bullet} : \partial(Q,f) \to K$
a border function
of $(Q,f)$.
We consider the quotient algebra
\[
  B(Q,f,m_{\bullet},b_{\bullet})
   = K Q / J(Q,f,m_{\bullet},b_{\bullet}),
\]
where $J(Q,f,m_{\bullet},b_{\bullet})$
is the ideal in the path algebra $KQ$ of $Q$ over $K$
generated by the elements:
\begin{enumerate}[(1)]
 \item
  $\alpha f({\alpha})$,
  for all arrows $\alpha \in Q_1$ which are not border loops,
 \item
  $\alpha^2 - b_{s(\alpha)} B_{\alpha}$,
  for all border loops $\alpha \in Q_1$,
 \item
  $B_{\alpha}
   - B_{\bar{\alpha}}$,
  for all arrows $\alpha \in Q_1$.
\end{enumerate}
Then $B(Q,f,m_{\bullet},b_{\bullet})$ is
said to be a
\emph{socle deformed biserial weighted triangulation algebra}.
Moreover, if  $(Q,f) = (Q(S,\vv{T}),f)$
for a directed triangulated surface $(S,\vv{T})$,
then
$B(Q(S,\vv{T}),f,m_{\bullet},b_{\bullet})$
is said to be a \emph{socle deformed biserial weighted surface algebra},
and is denoted by
$B(S,\vv{T},m_{\bullet},b_{\bullet})$.

\begin{proposition}
\label{prop:6.2}
Let $(Q,f)$ be a triangulation quiver
with $\partial(Q,f)$
not empty,
$m_{\bullet}$,
$b_{\bullet}$
weight
and
border functions of $(Q,f)$,
$\bar{B} = B(Q,f,m_{\bullet},b_{\bullet})$,
and
$B = B(Q,f,m_{\bullet})$.
Then the following statements hold.
\begin{enumerate}[(i)]
 \item
  $\bar{B}$ is  finite-dimensional biserial 
  with $\dim_K \bar{B} = \sum_{\cO \in \cO(g)} m_{\cO} n_{\cO}^2$.
 \item
  $\bar{B}$ is a symmetric algebra.
 \item
  $\bar{B}$ is socle equivalent to $B$.
 \item
  $\bar{B}$ degenerates to $B$.
 \item
  $\bar{B}$ is a tame algebra.
 \item
  $\bar{B}$ is an algebra of generalized dihedral type.
\end{enumerate}
\end{proposition}

\begin{proof}
Write $\bar{J} = J(Q,f,m_{\bullet},b_{\bullet})$.

\smallskip

(i)
Let $i$ be the vertex of $Q$ and $\alpha$, $\bar{\alpha}$
the two arrows in $Q$ with source $i$.
Then the indecomposable projective right $\bar{B}$-module
$P_i = e_i \bar{B}$ at the vertex $i$ 
has a basis given by $e_i$,
the cosets $u+\bar{J}$ of all initial proper subwords $u$
of $B_{\alpha}$ and $B_{\bar{\alpha}}$,
and $B_{\alpha}+\bar{J} = B_{\bar{\alpha}}+\bar{J}$,
and hence
$\dim_K P_i = m_{\alpha} n_{\alpha} + m_{\bar{\alpha}} n_{\bar{\alpha}}$.
Then we obtain
\[
  \dim_K \bar{B} = \sum_{\cO \in \cO(g)} m_{\cO} n_{\cO}^2 .
\]
Clearly, $\bar{B}$ is a biserial algebra.

\smallskip

(ii)
We define a  symmetrizing form $\bar{\varphi} : \bar{B} \to K$
of $\bar{B} = K Q / \bar{J}$ 
by assigning to the coset $u+\bar{J}$ of a path $u$ in $Q$
the following element
\[
  \bar{\varphi}(u+\bar{J}) = \left\{
   \begin{array}{cl}
   1 & \mbox{if $u = B_{\alpha}$ for an arrow $\alpha \in Q_1$},\\
   b_i & \mbox{if $u = \alpha^2$ for some border loop $\alpha \in Q_1$},\\
   0 & \mbox{otherwise},
   \end{array}
  \right.
\]
and extending to a $K$-linear form.

\smallskip

(iii)
The algebras
$B/\soc(B)$ and $\bar{B}/\soc(\bar{B})$
are isomorphic to the algebra $K Q / I$,
where $I$ is the ideal in $K Q$ generated by the elements
$\alpha f(\alpha)$ and $B_{\alpha}$,
for all arrows $\alpha$ in $Q_1$.
Hence $B$ and $\bar{B}$ are socle equivalent.

\smallskip

(iv)
For each $t \in K$,
consider the quotient algebra
$\bar{B}(t) = KQ/\bar{J}^{(t)}$,
where $\bar{J}^{(t)}$ is the ideal
in $K Q$ generated by the elements:
\begin{enumerate}[(1)]
 \item
  ${\alpha} f({\alpha})$,
  for all arrows $\alpha \in Q_1$ which are not border loops,
 \item
  $\alpha^2 - t b_{s(\alpha)} B_{\bar{\alpha}}$,
  for all border loops $\alpha \in Q_1$,
 \item
  $B_{{\alpha}} - B_{\bar{\alpha}}$,
  for all arrows $\alpha \in Q_1$.
\end{enumerate}
Then $\bar{B}(t)$, $t \in K$,
is an algebraic family in the variety $\alg_{d}(K)$,
with $d = \dim_K \bar{B}$,
such that
$\bar{B}(t) \cong \bar{B}(1) = \bar{B}$
for all $t \in K^*$ and
$\bar{B}(0) \cong B$.
Then it follows from Proposition~\ref{prop:6.1}
that $\bar{B}$ degenerates to $B$.

\smallskip

(v)
$\bar{B}$ is a tame algebra, since it
is socle equivalent to the tame algebra $B$.

\smallskip

(vi)
Since $B$ and $\bar{B}$ are socle equivalent,
their  stable Auslander-Reiten quivers
$\Gamma_B^s$ and $\Gamma_{\bar{B}}^s$ 
are isomorphic as  translation quivers.
Moreover,
$\tau_B = \Omega_B^2$ and $\tau_{\bar{B}} = \Omega_{\bar{B}}^2$,
because $B$ and $\bar{B}$ are symmetric.
Finally, we observe that the actions of the
syzygy operators $\Omega_B$ and $\Omega_{\bar{B}}$
on the uniserial modules $U(\alpha)$, for $\alpha \in Q_1$
with $f(\alpha) \neq \alpha$,
forming the mouth of stable tubes of rank $3$ in
$\Gamma_B^s = \Gamma_{\bar{B}}^s$ coincide.
Therefore, $\bar{B}$ is an algebra of generalized
dihedral type.
\end{proof}

Let $(Q,f)$ be a triangulation quiver,
$\alpha$ an arrow in $Q_1$, and $i = s(\alpha)$.
We define also the path

$A_{\alpha} = \big(\alpha g(\alpha) \dots 
    g^{n_{\alpha} - 1}(\alpha)\big)^{m_{\alpha} - 1}
    \alpha g(\alpha) \dots g^{n_{\alpha} - 2}(\alpha)$,
if $n_{\alpha} \geq 2$,

$A_{\alpha} = \alpha^{m_{\alpha} - 1}$,
if $n_{\alpha} = 1$ and $m_{\alpha} \geq 2$,

$A_{\alpha} = e_i$,
if $n_{\alpha} = 1$ and $m_{\alpha} = 1$,

\noindent
in $Q$ of length 
$m_{\alpha} n_{\alpha} - 1$ from $i = s(\alpha)$ 
to $t(g^{n_{\alpha} - 2}(\alpha))$.
Hence, $B_{\alpha} = A_{\alpha} g^{n_{\alpha} - 1}(\alpha)$.

\begin{proposition}
\label{prop:6.3}
Let $(Q,f)$ be a triangulation quiver
with non-empty border $\partial(Q,f)$, and let 
$m_{\bullet}$ and 
$b_{\bullet}$ be 
weight
and
border functions of $(Q,f)$.
Assume that $K$ has characteristic different from $2$.
Then the algebras
$B(Q,f,m_{\bullet},b_{\bullet})$
and
$B(Q,f,m_{\bullet})$
are isomorphic.
\end{proposition}

\begin{proof}
Since $K$ has characteristic different from $2$,
for any vertex $i \in \partial(Q,f)$ there exists
a unique element $a_i \in K$ such that $b_i = 2 a_i$.
Then there exists an isomorphism of $K$-algebras
$h :B(Q,f,m_{\bullet}) \to  B(Q,f,m_{\bullet},b_{\bullet})$
such that
\[
   h(\alpha) = \left\{ \begin{array}{cl}
      \alpha & \mbox{for any arrow $\alpha \in Q_1$ which is not a border loop}, \\
      \alpha - a_{s(\alpha)}A_{\bar{\alpha}} & \mbox{for any border loop $\alpha \in Q_1$}.
   \end{array} \right.
\]
We note that, if $\alpha \in Q_1$ is a border loop, then
the following equalities
\begin{align*}
  \alpha A_{\bar{\alpha}} &= B_{\alpha} = B_{\bar{\alpha}} = A_{\bar{\alpha}} \alpha,
&  
  \big(\alpha - a_{s(\alpha)} A_{\bar{\alpha}}\big)^2 &= 0  
\end{align*}
hold in
$B(Q,f,m_{\bullet},b_{\bullet})$.
\end{proof}

The following proposition is a special case of
\cite[Theorem~5.3]{ESk6}.

\begin{proposition}
\label{prop:6.4}
Let $A$ be a basic, indecomposable,
symmetric algebra with the Grothendieck group $K_0(A)$
of rank at least $2$ which is socle equivalent to
a biserial weighted triangulation algebra
$B(Q,f,m_{\bullet})$.
\begin{enumerate}[(i)]
 \item
  If $\partial(Q,f)$ is empty then $A$ is isomorphic to
  $B(Q,f,m_{\bullet})$.
 \item
  Otherwise $A$ is isomorphic to $B(Q,f,m_{\bullet},b_{\bullet})$
  for some border function $b_{\bullet}$ of $(Q,f)$.
\end{enumerate}
\end{proposition}

The following example shows that
a  socle deformed biserial weighted triangulation algebra
need  not be isomorphic to a biserial weighted triangulation algebra.

\begin{example}
\label{ex:6.5}
Let $(Q,f)$ be
the triangulation quiver
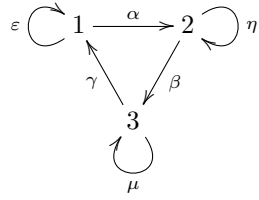
\[
  \xymatrix@C=.8pc{
     1 \ar@(dl,ul)[]^{\varepsilon} \ar[rr]^{\alpha} &&
     2 \ar@(ur,dr)[]^{\eta} \ar[ld]^{\beta} \\
     & 3 \ar@(dr,dl)[]^{\mu} \ar[lu]^{\gamma}}
\]
with $f$-orbits
$(\alpha\, \beta\, \gamma)$,
$(\varepsilon)$,
$(\eta)$,
$(\mu)$,
considered in Example~\ref{ex:4.3} .
Then $\cO(g)$ consists of one $g$-orbit
$(\alpha\, \eta\, \beta\, \mu\, \gamma\, \varepsilon)$.
Let
$m_{\bullet} : \cO(g) \to \bN$
be the weight function
with
$m_{\cO(\alpha)} = 1$.
Then the associated biserial weighted triangulation algebra
$B = B(Q,f,m_{\bullet})$
is given by the above quiver and the relations
\begin{align*}
 \alpha\beta &= 0,
 &
 \varepsilon^2 &= 0,
 & 
 \alpha\eta\beta\mu\gamma\varepsilon &= \varepsilon\alpha\eta\beta\mu\gamma ,
\\
 \beta\gamma &= 0,
 &
 \eta^2 &= 0,
 &
 \beta\mu\gamma\varepsilon\alpha\eta &= \eta\beta\mu\gamma\varepsilon\alpha ,
\\
 \gamma\alpha &= 0,
 &
 \mu^2 &= 0,
 &
 \gamma\varepsilon\alpha\eta\beta\mu &= \mu\gamma\varepsilon\alpha\eta\beta.
\end{align*}
Observe that the border
$\partial(Q,f)$
of $(Q,f)$
is the set $Q_0 = \{1,2,3\}$ of vertices of $Q$,
and $\varepsilon$, $\eta$, $\mu$
are the border loops.
Take now a border function
$b_{\bullet} : \partial(Q,f) \to K$.
Then the associated socle  deformed
biserial weighted triangulation algebra
$\bar{B} = B(Q,f,m_{\bullet},b_{\bullet})$
is given by the above quiver and the relations
\begin{align*}
 \alpha\beta &= 0,
 &
 \varepsilon^2 &= b_1 \varepsilon\alpha\eta\beta\mu\gamma,
 & 
 \alpha\eta\beta\mu\gamma\varepsilon &= \varepsilon\alpha\eta\beta\mu\gamma ,
\\
 \beta\gamma &= 0,
 &
 \eta^2 &= b_2 \eta\beta\mu\gamma\varepsilon\alpha ,
 &
 \beta\mu\gamma\varepsilon\alpha\eta &= \eta\beta\mu\gamma\varepsilon\alpha ,
\\
 \gamma\alpha &= 0,
 &
 \mu^2 &= b_3 \mu\gamma\varepsilon\alpha\eta\beta ,
 &
 \gamma\varepsilon\alpha\eta\beta\mu &= \mu\gamma\varepsilon\alpha\eta\beta.
\end{align*}

Assume that $K$ has characteristic $2$ and $b_{\bullet}$
is non-zero, say $b_1 \neq 0$.
We claim that the algebras
${B}$ and $\bar{B}$ are not isomorphic.
Suppose that the
algebras $B$ and $\bar{B}$ are isomorphic.
Then there is an isomorphism
$h : {B} \to \bar{B}$ of $K$-algebras
such that $h(e_i) = e_i$ for any $i \in \{1,2,3\}$.
In particular, we conclude that
$h(\varepsilon) \in e_1 \bar{B} e_1$.
Observe that the $K$-vector space $e_1 \bar{B} e_1$
has the basis 
$\varepsilon$, 
$\alpha\eta\beta\mu\gamma$,
$\alpha\eta\beta\mu\gamma\varepsilon = \varepsilon\alpha\eta\beta\mu\gamma$.
Hence 
$h(\varepsilon) = u_1 \varepsilon + u_2 \alpha\eta\beta\mu\gamma
  + u_3 \alpha\eta\beta\mu\gamma\varepsilon$
for some $u_1 \in K^*$ and $u_2, u_3 \in K$.
Since $K$ is of characteristic $2$ and $(\rad \bar{B})^7 = 0$,
we conclude that the following equalities hold in $\bar{B}$
\[
0 = h(\varepsilon^2) = h(\varepsilon)^2 
  = u_1^2 \varepsilon^2 + u_1 u_2 \varepsilon \alpha\eta\beta\mu\gamma
    + u_1 u_2 \alpha\eta\beta\mu\gamma\varepsilon
  = u_1^2 b_1 \varepsilon \alpha\eta\beta\mu\gamma
,
\]
and hence
$u_1^2 b_1 = 0$, a contradiction.
This shows that the algebras
${B}$ and $\bar{B}$ are not isomorphic.
We also note that $\bar{B}$ is a biserial
but not a special biserial algebra.
\end{example}

\section{Idempotent algebras of biserial weighted triangulation algebras}\label{sec:idempalg}

The aim of this section is to introduce the idempotent algebras
of biserial weighted surface algebras occurring in
Theorem~\ref{th:main1}.

Let $(Q,f)$ be a triangulation quiver.
A \emph{$2$-triangle disk in $(Q,f)$} is a subquiver $D$ of $(Q,f)$ 
of the form
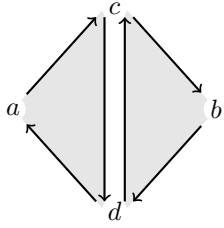
\[
\begin{tikzpicture}
[->,scale=.9]
\coordinate (1) at (0,1.5);
\coordinate (1l) at (-0.15,1.5);
\coordinate (1p) at (0.15,1.5);
\coordinate (2) at (0,-1.5);
\coordinate (2l) at (-0.15,-1.5);
\coordinate (2p) at (0.15,-1.5);
\coordinate (3) at (-1.5,0);
\coordinate (4) at (1.5,0);
\fill[fill=gray!20] (1l) -- (2l) -- (3) -- cycle;
\fill[fill=gray!20] (1p) -- (2p) -- (4) -- cycle;
\node[fill=white,circle,minimum size=3]  (1) at (0,1.5) { };
\node (1) at (0,1.5) {$c$};
\node (1l) at (-0.15,1.5) { };
\node (1p) at (0.15,1.5) { };
\node[fill=white,circle,minimum size=3]  (2) at (0,-1.5) { };
\node (2) at (0,-1.5) {$d$};
\node (2l) at (-0.15,-1.5) { };
\node (2p) at (0.15,-1.5) { };
\node[fill=white,circle,minimum size=3]  (3) at (-1.5,0) { };
\node (3) at (-1.5,0) {$a$};
\node[fill=white,circle,minimum size=3]  (4) at (1.5,0) { };
\node (4) at (1.5,0) {$b$};

\fill[fill=gray!20] (1l) -- (2l) -- (3) -- cycle;
\fill[fill=gray!20] (1p) -- (2p) -- (4) -- cycle;

\draw[thick,->]
(1l) edge (2l)
(2l) edge (3)
(3) edge (1l)
(2p) edge (1p)
(1p) edge (4)
(4) edge (2p)
;
\end{tikzpicture}
\]
where
the shaded triangles describe  $f$-orbits.
We note that $D$ may be obtained as follows: Take  the triangulation quiver
associated to the following triangulation of the disk
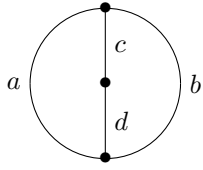
\[
\begin{tikzpicture}[auto]
\coordinate (c) at (0,0);
\coordinate (u) at (0,1);
\coordinate (d) at (0,-1);
\coordinate (r) at (1,0);
\coordinate (l) at (-1,0);
\draw (u) to node {$c$} (c);
\draw (c) to node {$d$} (d);
\draw (r) node [right] {$b$} arc (0:180:1) node [left] {$a$};
\draw (l) arc (180:360:1);
\node (u) at (0,1) {$\bullet$};
\node (d) at (0,-1) {$\bullet$};
\node (c) at (0,0) {$\bullet$};
\end{tikzpicture}
\]
with boundary edges $a$ and $b$, and the coherent orientation
of triangles 
$(a \, c \, d)$, 
$(b \, d \, c)$, 
then remove  the loops given by $a$ and $b$.

Let $(Q,f)$ be a triangulation quiver, 
$m_{\bullet} : \cO(g) \to \bN^*$
a weight function of $(Q,f)$,
and $\Sigma $ be a collection 
of $2$-triangle disks in $(Q,f)$.
We denote by $e_{\Sigma}$
the idempotent of the algebra
$B(Q,f,m_{\bullet})$
which is  the sum of all primitive idempotents
corresponding to all vertices of $(Q,f)$
excluding  the $2$-cycle vertices of the $2$-triangle 
disks from $\Sigma$, that is the vertices $c, d$ in the above diagram.
We define
\[
  B(Q,f,\Sigma,m_{\bullet}):
   =  e_{\Sigma} B(Q,f,m_{\bullet}) e_{\Sigma}
\]
and we call this  the \emph{idempotent algebra} of
$B(Q,f,m_{\bullet})$
with respect to $\Sigma$.
We note that if $\Sigma$ is empty then
$B(Q,f,\Sigma,m_{\bullet}) = B(Q,f,m_{\bullet})$.
On the other hand,
if $\Sigma$ is not empty, then every $2$-triangle disk $D$
from $\Sigma$ is replaced in the quiver of
$B(Q,f,\Sigma,m_{\bullet})$
by the $2$-cycle
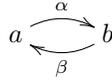
\[
  \xymatrix{ a \ar@/^1.5ex/[r]^{\alpha} & b \ar@/^1.5ex/[l]^{\beta}}
\]
with 
$\alpha \beta = 0$
and
$\beta \alpha = 0$
in 
$B(Q,f,\Sigma,m_{\bullet})$.
If $(Q,f)$ is the triangulation quiver
$(Q(S,\vv{T}),f)$
associated  to a directed triangulated surface
$(S,\vv{T})$,
then the idempotent algebra
$B(Q(S,\vv{T}),f,\Sigma,m_{\bullet})$
is called the \emph{idempotent biserial weighted surface algebra},
with respect to $\Sigma$, and denoted by
$B(S,\vv{T},\Sigma,m_{\bullet})$.

The following proposition describes basic properties 
of these idempotent algebras.

\begin{proposition}
\label{prop:7.1}
Let $(Q,f)$ be a triangulation quiver, 
$m_{\bullet} : \cO(g) \to \bN^*$
a weight function of $(Q,f)$,
$\Sigma $ a non-empty collection 
of $2$-triangle disks in $(Q,f)$,
and $B = B(Q,f,\Sigma,m_{\bullet})$
the associated idempotent algebra.
Then the following statements hold.
\begin{enumerate}[(i)]
 \item
  ${B}$ is  finite-dimensional special biserial.
 \item
  ${B}$ is a tame symmetric algebra.
 \item
  ${B}$ is an algebra of generalized dihedral type.
 \item
  The Cartan matrix $C_{{B}}$ of ${B}$ is singular.
 \item
  For each $2$-cycle 
  $\xymatrix{ a \ar@/^1.5ex/[r]^{\alpha} & b \ar@/^1.5ex/[l]^{\beta}}$
  in the Gabriel quiver $Q_B$
  given by a $2$-triangle disk from $\Sigma$,
  the uniserial modules 
  $\alpha B$ and $\beta B$ are periodic of period $2$ such that
  $\Omega_B(\alpha B) = \beta B$ 
  and
  $\Omega_B(\beta B) = \alpha B$, 
  and hence they
  lie on the mouth of two different stable tubes
  of $\Gamma_B$ of rank $1$.
\end{enumerate}
\end{proposition}

\begin{proof}
Statements (i) and (ii) follow from 
Propositions \ref{prop:2.1} and \ref{prop:2.4}.

Let
$\xymatrix{ a \ar@/^1.5ex/[r]^{\alpha} & b \ar@/^1.5ex/[l]^{\beta}}$
be a $2$-cycle in $Q_B$
given by a $2$-triangle disk from $\Sigma$.
Then we have
$\alpha \beta = 0$ and $\beta \alpha = 0$ in $B$, 
and direct checking shows that
$\Omega_B(\alpha B) = \beta B$ 
and
$\Omega_B(\beta B) = \alpha B$.
Clearly,  $\alpha B$ and $\beta B$ are periodic modules
of period $2$
lying on the mouth of two different stable tubes
of $\Gamma_B$ of rank $1$.
Moreover, we have
$[e_a B] = [e_b B]$ in $K_0 (B)$,
and hence $C_B$ is singular.
This proves  statements (iv) and (v).
Finally, the statement (iii) follows from 
Proposition~\ref{prop:5.1}
and its proof.
We note that every arrow in $Q_B$ which does
not belong to a $2$-cycle given by a $2$-triangle
disk from $\Sigma$ belongs to an $f$-orbit
of length $1$ or $3$.
\end{proof}

Let $(Q,f)$ be a triangulation quiver
with non-empty border $\partial(Q,f)$, 
$m_{\bullet} : \cO(g) \to \bN^*$
a weight function of $(Q,f)$,
$b_{\bullet} : \partial(Q,f) \to K$ a border function of $(Q,f)$,
and $\Sigma $ a collection 
of $2$-triangle disks in $(Q,f)$.
We denote also by $e_{\Sigma}$
the idempotent in the socle deformed biserial weighted
triangulation algebra
$B(Q,f,m_{\bullet},b_{\bullet})$
being the sum of all primitive idempotents
corresponding to all vertices of $(Q,f)$
except the $2$-cycle vertices of the $2$-triangle 
disks from $\Sigma$.
Then
\[
  B(Q,f,\Sigma,m_{\bullet},b_{\bullet})
   =  e_{\Sigma} B(Q,f,m_{\bullet},b_{\bullet}) e_{\Sigma}
\]
is said to be the \emph{idempotent algebra} of
$B(Q,f,m_{\bullet},b_{\bullet})$,
with respect to $\Sigma$.

The following proposition describes basic properties 
of the idempotent algebras introduced above.

\begin{proposition}
\label{prop:7.2}
Let $(Q,f)$ be a triangulation quiver 
with $\partial(Q,f)$ non-empty,
$m_{\bullet}$,
$b_{\bullet}$
weight and border functions of $(Q,f)$,
$\Sigma $ a non-empty collection 
of $2$-triangle disks in $(Q,f)$,
and
$B = B(Q,f,\Sigma,m_{\bullet})$,
$\bar{B} = B(Q,f,\Sigma,m_{\bullet},b_{\bullet})$.
Then the following statements hold.
\begin{enumerate}[(i)]
 \item
  $\bar{B}$ is  finite-dimensional biserial.
 \item
  $\bar{B}$ is a symmetric algebra.
 \item
  $\bar{B}$ is socle equivalent to $B$.
 \item
  $\bar{B}$ degenerates to $B$.
 \item
  $\bar{B}$ is a tame algebra.
 \item
  $\bar{B}$ is an algebra of generalized dihedral type.
 \item
  The Cartan matrix $C_{\bar{B}}$ of $\bar{B}$ is singular.
\end{enumerate}
Moreover, if $K$ is of characteristic different from $2$, then
\begin{enumerate}[(i)]
 \addtocounter{enumi}{7}
  \item
  $\bar{B}$ is isomorphic to $B$.
\end{enumerate}
\end{proposition}

\begin{proof}
The required statements follow from 
Propositions 
\ref{prop:6.2},
\ref{prop:6.3},
\ref{prop:7.1},
and their proofs.
\end{proof}

The following proposition is a special case of
\cite[Theorem~5.3]{ESk6}.

\begin{proposition}
\label{prop:7.3}
Let $A$ be a basic, indecomposable, symmetric algebra
with the Grothendieck group $K_0(A)$ of rank at least $2$
which is socle equivalent to the idempotent algebra
$B(Q,f,\Sigma,m_{\bullet})$ 
for a non-empty collection $\Sigma$ of $2$-triangle disks
and a weight function $m_{\bullet}$ of $(Q,f)$.
\begin{enumerate}[(i)]
 \item
  If $\partial(Q,f)$ is empty then $A$ is isomorphic to
  $B(Q,f,\Sigma,m_{\bullet})$.
 \item
  Otherwise $A$ is isomorphic to $B(Q,f,\Sigma,m_{\bullet},b_{\bullet})$
  for some border function $b_{\bullet}$ of $(Q,f)$.
\end{enumerate}
\end{proposition}

\begin{example}
\label{ex:7.4} 
Let $S$ be the connected sum $\bT \# D$
(see \cite[Section~3.1]{Ca})
of the torus $\bT$ and the disk $D = D^2$, 
and $T$ be the following triangulation of $S$
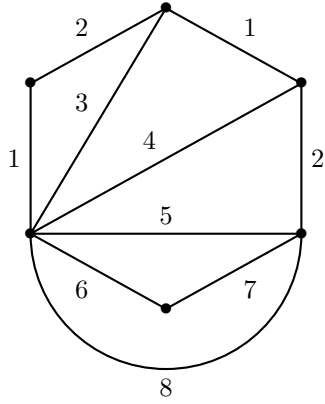
\[
\begin{tikzpicture}
[scale=1,auto]
\node (1) at (-1.8,-1) {$\bullet$};
\node (2) at (-1.8,1) {$\bullet$};
\node (3) at (0,2) {$\bullet$};
\node (4) at (1.8,1) {$\bullet$};
\node (5) at (1.8,-1) {$\bullet$};
\node (6) at (0,-2) {$\bullet$};
\coordinate (1) at (-1.8,-1);
\coordinate (2) at (-1.8,1) ;
\coordinate (3) at (0,2) ;
\coordinate (4) at (1.8,1) ;
\coordinate (5) at (1.8,-1) ;
\coordinate (6) at (0,-2) ;
\draw[thick] (1) arc (180:360:1.8) node[midway,below] {8};
\draw[thick]
(1) edge node {1} (2) 
(1) edge node {3} (3) 
(1) edge node {4} (4) 
(1) edge node {5} (5) 
(2) edge node {2} (3) 
(3) edge node {1} (4) 
(4) edge node {2} (5) 
(5) edge node {7} (6) 
(6) edge node {6} (1) ;
\end{tikzpicture}
\]

Consider the following orientation $\vv{T}$ of triangles in $T$ 
\[
\mbox{
  (1 2 3), (1 3 4), (2 4 5), (5 6 7), (8 7 6).
}
\]
Then the associated triangulation quiver $(Q(S,\vv{T}),f)$ 
is of the form
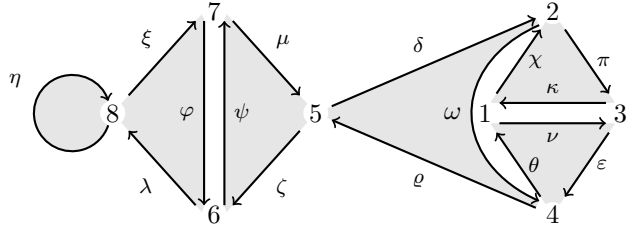
\[
\begin{tikzpicture}
[->,scale=.9]
\coordinate (7) at (0,1.5);
\coordinate (7l) at (-0.15,1.5);
\coordinate (7p) at (0.15,1.5);
\coordinate (6) at (0,-1.5);
\coordinate (6l) at (-0.15,-1.5);
\coordinate (6p) at (0.15,-1.5);
\coordinate (8) at (-1.5,0);
\coordinate (5) at (1.5,0);

\coordinate (1) at (4,0);
\coordinate (1u) at (4.075,.15);
\coordinate (1d) at (4.075,-.15);
\coordinate (3) at (6,0);
\coordinate (3u) at (5.925,.15);
\coordinate (3d) at (5.925,-.15);
\coordinate (2) at (5,1.5);
\coordinate (4) at (5,-1.5);

\coordinate (2l) at (4.75,1.4);
\coordinate (4l) at (4.75,-1.4);

\fill[fill=gray!20] (1u) -- (2) -- (3u) -- cycle;
\fill[fill=gray!20] (3d) -- (4) -- (1d) -- cycle;

\fill[fill=gray!20] (7l) -- (6l) -- (8) -- cycle;
\fill[fill=gray!20] (7p) -- (6p) -- (5) -- cycle;

\fill[fill=gray!20] (4.8,1.3) arc [start angle=120, delta angle=120, x radius=2, y radius=1.5] -- (4l) -- (5) -- (2l) -- cycle;

\node [circle,minimum size=1.5](A) at (-1.5,0) { };
\node [circle,minimum size=1cm](B) at (-2.1,0) {};
\coordinate  (C) at (intersection 2 of A and B);
\coordinate  (D) at (intersection 1 of A and B);
 \tikzAngleOfLine(B)(D){\AngleStart}
 \tikzAngleOfLine(B)(C){\AngleEnd}
\fill[gray!20]%
   let \p1 = ($ (B) - (D) $), \n2 = {veclen(\x1,\y1)}
   in   
     (D) arc (\AngleStart:\AngleEnd:\n2); 
\node [fill=white,circle,minimum size=1.5](A) at (-1.5,0) { };
\draw[thick,<-]%
   let \p1 = ($ (B) - (D) $), \n2 = {veclen(\x1,\y1)}
   in   
     (B) ++(60:\n2) node[left]{\footnotesize\raisebox{0ex}{$\!\!\!\!\!\!\!\!\!\!\!\!\!\!\!\eta\ \ \ \ \ \ \ \ $}}
     (D) arc (\AngleStart:\AngleEnd:\n2); 
\node[fill=white,circle,minimum size=3]  (7) at (0,1.5) { };
\node (7) at (0,1.5) {7};
\node (7l) at (-0.15,1.5) { };
\node (7p) at (0.15,1.5) { };
\node[fill=white,circle,minimum size=3]  (6) at (0,-1.5) { };
\node (6) at (0,-1.5) {6};
\node (6l) at (-0.15,-1.5) { };
\node (6p) at (0.15,-1.5) { };
\node[fill=white,circle,minimum size=3]  (8) at (-1.5,0) { };
\node (8) at (-1.5,0) {8};
\node[fill=white,circle,minimum size=3]  (5) at (1.5,0) { };
\node (5) at (1.5,0) {5};

\node[fill=white,circle,minimum size=3]  (1) at (4,0) { };
\node[fill=white,circle,minimum size=3]  (3) at (6,0) { };
\node[fill=white,circle,minimum size=3]  (2) at (5,1.5) { };
\node[fill=white,circle,minimum size=3]  (4) at (5,-1.5) { };

\node (1) at (4,0) {1};
\node (3) at (6,0) {3};
\node (2) at (5,1.5) {2};
\node (4) at (5,-1.5) {4};

\node (1u) at (4.075,.15) { };
\node (1d) at (4.075,-.15) { };
\node (3u) at (5.925,.15) { };
\node (3d) at (5.925,-.15) { };

\fill[fill=gray!20] (7l) -- (6l) -- (8) -- cycle;
\fill[fill=gray!20] (7p) -- (6p) -- (5) -- cycle;

\draw[thick] (4.8,1.3) arc [start angle=120, delta angle=120, x radius=2, y radius=1.5] node[midway,left] {$\omega$};

\draw[thick,->]
(7l) edge node[left]{\footnotesize$\varphi$} (6l)
(6l) edge node[below left]{\footnotesize$\lambda$} (8)
(8) edge node[above left]{\footnotesize$\xi$} (7l)
(6p) edge node[right]{\footnotesize$\psi$} (7p)
(7p) edge node[above right]{\footnotesize$\mu$} (5)
(5) edge node[below right]{\footnotesize$\zeta$} (6p)
(1u) edge node[right]{\footnotesize$\chi$} (2)
(2) edge node[right]{\footnotesize$\pi$} (3u)
(3u) edge node[above]{\footnotesize$\kappa$} (1u)
(1d) edge node[below]{\footnotesize$\nu$} (3d)
(3d) edge node[right]{\footnotesize$\varepsilon$} (4)
(4) edge node[right]{\footnotesize$\theta$} (1d)
(5) edge node[above left]{\footnotesize$\delta$} (2l)
;
\draw[thick,->] (4l) -- node[below left]{\footnotesize$\varrho$} (5);
\end{tikzpicture}
\]
where the shaded triangles denote $f$-orbits.
We set $Q = Q(S,\vv{T})$.
The set $\cO(g)$ of $g$-orbits of $(Q,f)$
consists of the $g$-orbits
\begin{align*}
 &&
 \cO(\eta) &= ( \eta\, \xi\, \mu\, \delta\, \pi\, \varepsilon\, \varrho\, \zeta\, \lambda ), &
 \cO(\omega) &= ( \omega\, \theta\, \chi ), &
 \cO(\kappa) &= ( \kappa\, \nu ), &
 \cO(\varphi) &= ( \varphi\, \psi ). 
 &&
\end{align*}
Hence
a weight function
$m_{\bullet} : \cO(g) \to \bN^*$
is given by four positive natural numbers
\begin{align*}
&&
 m &= m_{\cO(\eta)},&
 n &= m_{\cO(\omega)},&
 p &= m_{\cO(\kappa)},&
 q &= m_{\cO(\varphi)}.&
&&
\end{align*}
Then the associated biserial weighted
triangulation algebra
$B(Q,f,m_{\bullet})$
is given by the quiver $Q$ and the relations
\begin{gather*}
  \eta^2 = 0,
\ \ 
  \xi\varphi = 0,
\ \ 
  \varphi\lambda = 0,
\ \ 
  \lambda\xi = 0,
\ \ 
  \zeta\psi = 0,
\ \ 
  \psi\mu = 0,
\ \ 
  \mu\zeta = 0,
\ \ 
  \delta\omega = 0,
  \\
  \omega\varrho = 0,
\ \ 
  \varrho\delta = 0,
\ \ 
  \chi\pi = 0,
\ \ 
  \pi\kappa = 0,
\ \ 
  \kappa \chi = 0,
\ \ 
  \theta\nu = 0,
\ \ 
  \nu\varepsilon = 0,
\ \ 
  \varepsilon\theta = 0,
 \\
  (\chi\omega\theta)^n = (\nu\kappa)^p,
\quad
  (\pi\varepsilon\varrho\zeta\lambda\eta\xi\mu\delta)^m = (\omega\theta\chi)^n,
\quad
  (\kappa\nu)^p = (\varepsilon\varrho\zeta\lambda\eta\xi\mu\delta\pi)^m, 
 \\
  (\theta\chi\omega)^n = (\varrho\zeta\lambda\eta\xi\mu\delta\pi\varepsilon)^m, 
\qquad
  (\delta\pi\varepsilon\varrho\zeta\lambda\eta\xi\mu)^m
   = (\zeta\lambda\eta\xi\mu\delta\pi\varepsilon\varrho)^m,
 \\
  (\psi\varphi)^q = (\lambda\eta\xi\mu\delta\pi\varepsilon\varrho\zeta)^m,
\qquad
  (\varphi\psi)^q = (\mu\delta\pi\varepsilon\varrho\zeta\lambda\eta\xi)^m,
 \\
  (\xi\mu\delta\pi\varepsilon\varrho\zeta\lambda\eta)^m
   = (\eta\xi\mu\delta\pi\varepsilon\varrho\zeta\lambda)^m.
\end{gather*}
The triangulation quiver $(Q,f)$ contains two $2$-triangle
disks: $D_1$, given by the arrows
$\chi$, $\pi$, $\kappa$, $\nu$, $\varepsilon$, $\theta$, 
and $D_2$, given by the arrows
$\lambda$, $\xi$, $\varphi$, $\psi$, $\mu$, $\zeta$.
We note that $D_1$ is not a full subquiver of $Q$.
Let $\Sigma = \{D_1, D_2 \}$.
Then the idempotent algebra $B(Q,f,\Sigma,m_{\bullet})$
is given by the quiver $Q_{\Sigma}$ of the form
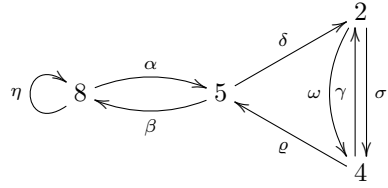
\[
 \xymatrix@C=3.5pc@R=1.5pc{
  && 2 \ar@<+.5ex>[dd]^{\sigma} \ar@<-.25ex>@/_2.5ex/[dd]_{\omega} \\
  8 \ar@(dl,ul)[]^{\eta} \ar@/^1.5ex/[r]^{\alpha}
   & 5 \ar@/^1.5ex/[l]^{\beta} \ar[ru]^{\delta}
 \\  && 4 \ar@<+.5ex>[uu]^{\gamma} \ar[lu]^{\varrho}
 }
\]
with 
$\alpha = \xi\mu$,
$\beta = \zeta\lambda$,
$\gamma = \theta\chi$,
$\sigma = \pi\varepsilon$,
and the relations
\begin{gather*}
  \eta^2 = 0,
\quad
  \alpha\beta = 0,
\quad
  \beta\alpha = 0,
\quad
  \delta\omega = 0,
\quad
  \omega\varrho = 0,
\quad
  \varrho\delta = 0,
\\
  \gamma\sigma = 0,
\quad
  \sigma\gamma = 0,
\quad
  (\sigma\varrho\beta\eta\alpha\delta)^m = (\omega\gamma)^n,
\quad
  (\gamma\omega)^n = (\varrho\beta\eta\alpha\delta\sigma)^m,
\\
  (\delta\sigma\varrho\beta\eta\alpha)^m 
   = (\beta\eta\alpha\delta\sigma\varrho)^m,
\quad
  (\alpha\delta\sigma\varrho\beta\eta)^m 
   = (\eta\alpha\delta\sigma\varrho\beta)^m.
\end{gather*}
The border $\partial(Q,f)$ of $(Q,f)$
consists of the vertex $8$, and hence
a border function 
$b_{\bullet} : \partial(Q,f) \to K$
is given by an element $b \in K$. 
Then the socle deformed biserial weighted
triangulation algebra
$B(Q,f,m_{\bullet},b_{\bullet})$
is obtained from 
$B(Q,f,m_{\bullet})$
by replacing 
the relation
$\eta^2 = 0$
by the relation
$\eta^2 = b(\eta\xi\mu\delta\pi\varepsilon\varrho\zeta\lambda)^m$.
Similarly, the algebra
$B(Q,f,\Sigma,m_{\bullet},b_{\bullet})$
is obtained from the algebra
$B(Q,f,\Sigma,m_{\bullet})$
by replacing 
the relation
$\eta^2 = 0$
by the relation
$\eta^2 = b(\eta\alpha\delta\sigma\varrho\beta)^m$.
\end{example}

\begin{example}
\label{ex:7.5} 
Let $n \geq 2$ be a natural number.
Consider the following triangulation $T(n)$
of the sphere $\bS = S^2$ in $\bR^3$
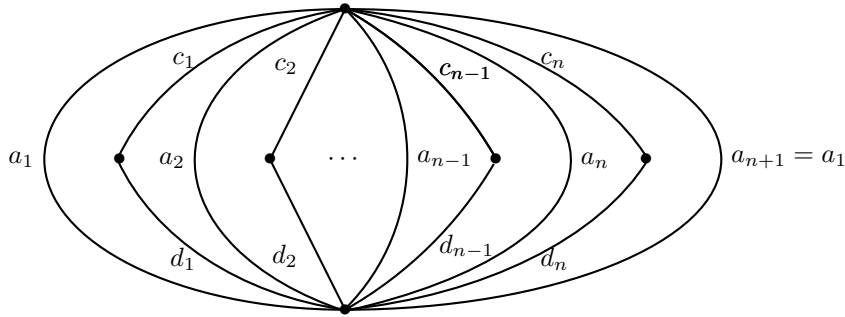
\[
\begin{tikzpicture}
[scale=1,auto]
\node (up) at (0,2) {$\bullet$};
\node (down) at (0,-2) {$\bullet$};
\node (1) at (-3,0) {$\bullet$};
\node (2) at (-1,0) {$\bullet$};
\node (3) at (2,0) {$\bullet$};
\node (4) at (4,0) {$\bullet$};
\node at (0,0) {$\cdots$};
\coordinate (up) at (0,2);
\coordinate (down) at (0,-2);
\coordinate (1) at (-3,0);
\coordinate (2) at (-1,0);
\coordinate (3) at (2,0);
\coordinate (4) at (4,0);
\draw[thick] (up) arc [start angle=90, delta angle=180, x radius=4, y radius=2] node[midway,left] {$a_1$};

\draw[thick] (up) arc [start angle=120, delta angle=120, x radius=4, y radius=2.32] node[midway,left] {$a_2$};

\draw[thick] (up) arc [start angle=45, delta angle=-90, x radius=2.82, y radius=2.82] node[midway,right] {$a_{n-1}$};

\draw[thick] (up) arc [start angle=60, delta angle=-120, x radius=6, y radius=2.32] node[midway,right] {$a_{n}$};

\draw[thick] (up) arc [start angle=90, delta angle=-180, x radius=5, y radius=2] node[midway,right] {$a_{n+1} = a_1$};

\draw[thick] (up) arc [start angle=60, delta angle=-30, x radius=5.4, y radius=5.3] node[midway,right] {$c_{n-1}$};

\draw[thick] (up) arc [start angle=60, delta angle=-30, x radius=5.4, y radius=5.3] node[midway,right] {$c_{n-1}$};

\draw[thick] (down) arc [start angle=-60, delta angle=30, x radius=5.4, y radius=5.3] node[midway,right] {$d_{n-1}$};

\draw[thick] (up) arc [start angle=80, delta angle=-60, x radius=5.25, y radius=3.1] node[midway,right] {$c_{n}$};

\draw[thick] (down) arc [start angle=-80, delta angle=60, x radius=5.25, y radius=3.1] node[midway,right] {$d_{n}$};

\draw[thick] (up) arc [start angle=100, delta angle=60, x radius=3.95, y radius=3.1] node[midway,left] {$c_{1}$};

\draw[thick] (down) arc [start angle=-100, delta angle=-60, x radius=3.95, y radius=3.1] node[midway,left] {$d_{1}$};

\draw[thick]
(down) edge node {$d_2$} (2) 
(2) edge node {$c_2$} (up) 
;
\end{tikzpicture}
\]
with two pole vertices and $n$ vertices lying on the equator.
Moreover, let $\vv{T(n)}$ be the coherent orientation
of the triangles of $T(n)$
\[
  (a_i \, c_i \, d_i),
 \quad 
  (d_i \, c_i \, a_{i+1}), 
 \quad 
  \mbox{ for }
  i \in \{ 1,\dots,n \} .
\]
Then the associated triangulation quiver $(Q(\bS,\vv{T(n)}),f)$ 
is of the form
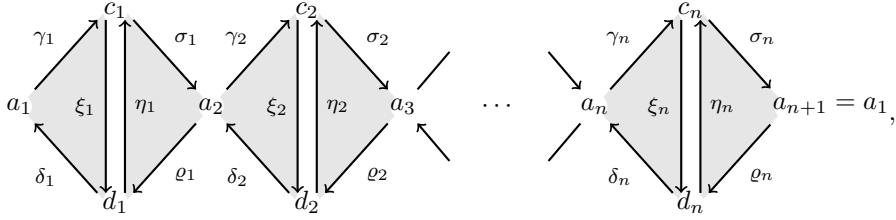
\[
\begin{tikzpicture}
[->,scale=.85]
\coordinate (0) at (-1.5,0);
\coordinate (1) at (0,1.5);
\coordinate (1l) at (-0.15,1.5);
\coordinate (1p) at (0.15,1.5);
\coordinate (2) at (0,-1.5);
\coordinate (2l) at (-0.15,-1.5);
\coordinate (2p) at (0.15,-1.5);
\coordinate (3) at (1.5,0);
\coordinate (4) at (3,1.5);
\coordinate (4l) at (2.85,1.5);
\coordinate (4p) at (3.15,1.5);
\coordinate (5) at (3,-1.5);
\coordinate (5l) at (2.85,-1.5);
\coordinate (5p) at (3.15,-1.5);
\coordinate (6) at (4.5,0);
\coordinate (7) at (6,1.5);
\coordinate (7l) at (5.85,1.5);
\coordinate (7p) at (6.15,1.5);
\coordinate (8) at (6,-1.5);
\coordinate (8l) at (5.85,-1.5);
\coordinate (8p) at (6.15,-1.5);
\coordinate (9) at (7.5,0);
\coordinate (10) at (9,1.5);
\coordinate (10l) at (8.85,1.5);
\coordinate (10p) at (9.15,1.5);
\coordinate (11) at (9,-1.5);
\coordinate (11l) at (8.85,-1.5);
\coordinate (11p) at (9.15,-1.5);
\coordinate (12) at (10.5,0);
\fill[fill=gray!20] (1l) -- (2l) -- (0) -- cycle;
\fill[fill=gray!20] (1p) -- (2p) -- (3) -- cycle;
\fill[fill=gray!20] (4l) -- (5l) -- (3) -- cycle;
\fill[fill=gray!20] (4p) -- (5p) -- (6) -- cycle;
\fill[fill=gray!20] (10l) -- (11l) -- (9) -- cycle;
\fill[fill=gray!20] (10p) -- (11p) -- (12) -- cycle;
\node[fill=white,circle,minimum size=3]  (0) at (-1.5,0) { };
\node (0) at (-1.5,0) {$a_1$};
\node[fill=white,circle,minimum size=3]  (1) at (0,1.5) { };
\node (1) at (0,1.5) {$c_1$};
\node (1l) at (-0.15,1.5) { };
\node (1p) at (0.15,1.5) { };
\node[fill=white,circle,minimum size=3]  (2) at (0,-1.5) { };
\node (2) at (0,-1.5) {$d_1$};
\node (2l) at (-0.15,-1.5) { };
\node (2p) at (0.15,-1.5) { };
\node[fill=white,circle,minimum size=3]  (3) at (1.5,0) { };
\node (3) at (1.5,0) {$a_2$};

\node[fill=white,circle,minimum size=3]  (4) at (3,1.5) { };
\node (4) at (3,1.5) {$c_2$};
\node (4l) at (2.85,1.5) { };
\node (4p) at (3.15,1.5) { };
\node[fill=white,circle,minimum size=3]  (5) at (3,-1.5) { };
\node (5) at (3,-1.5) {$d_2$};
\node (5l) at (2.85,-1.5) { };
\node (5p) at (3.15,-1.5) { };
\node[fill=white,circle,minimum size=3]  (6) at (4.5,0) { };
\node (6) at (4.5,0) {$a_3$};

\node (7l) at (5.35,1.) { };
\node (7p) at (6.65,1.) { };
\node (8l) at (5.35,-1.) { };
\node (8p) at (6.65,-1.) { };
\node at (6,0) {$\cdots$};

\node[fill=white,circle,minimum size=3]  (9) at (7.5,0) { };
\node (9) at (7.5,0) {$a_n$};
\node[fill=white,circle,minimum size=3]  (10) at (9,1.5) { };
\node (10) at (9,1.5) {$c_n$};
\node (10l) at (8.85,1.5) { };
\node (10p) at (9.15,1.5) { };
\node[fill=white,circle,minimum size=3]  (11) at (9,-1.5) { };
\node (11) at (9,-1.5) {$d_n$};
\node (11l) at (8.85,-1.5) { };
\node (11p) at (9.15,-1.5) { };
\node[fill=white,circle,minimum size=3]  (12) at (10.5,0) { };
\node (12) at (10.5,0) {$a_{n+1} = a_1\!\!\!\!\!\!\!\!\!\!\!\!\!\!\!\!\!\!\!\!$};

\fill[fill=gray!20] (1l) -- (2l) -- (0) -- cycle;
\fill[fill=gray!20] (1p) -- (2p) -- (3) -- cycle;

\draw[thick,->]
(1l) edge node[left]{\footnotesize$\xi_1$} (2l)
(2l) edge node[below left]{\footnotesize$\delta_1$} (0)
(0) edge node[above left]{\footnotesize$\gamma_1$} (1l)
(2p) edge node[right]{\footnotesize$\eta_1$} (1p)
(1p) edge node[above right]{\footnotesize$\sigma_1$} (3)
(3) edge node[below right]{\footnotesize$\varrho_1$} (2p)
(4l) edge node[left]{\footnotesize$\xi_2$} (5l)
(5l) edge node[below left]{\footnotesize$\delta_2$} (3)
(3) edge node[above left]{\footnotesize$\gamma_2$} (4l)
(5p) edge node[right]{\footnotesize$\eta_2$} (4p)
(4p) edge node[above right]{\footnotesize$\sigma_2$} (6)
(6) edge node[below right]{\footnotesize$\varrho_2$} (5p)
(8l) edge (6)
(7p) edge (9)
(10l) edge node[left]{\footnotesize$\xi_n$} (11l)
(11l) edge node[below left]{\footnotesize$\delta_n$} (9)
(9) edge node[above left]{\footnotesize$\gamma_n$} (10l)
(11p) edge node[right]{\footnotesize$\eta_n$} (10p)
(10p) edge node[above right]{\footnotesize$\sigma_n$} (12)
(12) edge node[below right]{\footnotesize$\varrho_n$} (11p)
;
\draw[thick,-]
(6) edge (7l)
(9) edge (8p)
;
\end{tikzpicture}
\qquad \quad \raisebox{9ex}{,}
\]
where the shaded triangles denote $f$-orbits.
The set $\cO(g)$ of $g$-orbits 
consists of the $2$-cycle orbits 
$\cO(\xi_i) = (\xi_i \, \eta_i)$,
$i \in \{ 1,\dots,n \}$,
and the two $n$-cycles
\begin{align*}
 \cO(\gamma_1) &=
  ( \gamma_1\, \sigma_1\, \gamma_2\, \sigma_2\, \dots \, \gamma_n\, \sigma_n ), &
 \cO(\varrho_n) &= 
  ( \varrho_n\, \delta_n\, \dots \, \varrho_2\, \delta_2\, \varrho_1\, \delta_1 ).
\end{align*}
Let
$m_{\bullet} : \cO(g) \to \bN^*$
be  a weight function, and
\begin{align*}
&&
 p &= m_{\cO(\gamma_1)},&
 q &= m_{\cO(\varrho_n)}.
&&
\end{align*}
Observe that the triangulation quiver 
$(Q(\bS,\vv{T(n)}),f)$
is formed by $2$-triangle disks
$D_1,D_2,\dots,D_n$.
Take $\Sigma = \{D_1,D_2,\dots,D_n\}$.
Then the associated idempotent algebra
$B(Q(\bS,\vv{T(n)}),f,\Sigma,m_{\bullet})$
is given by the quiver 
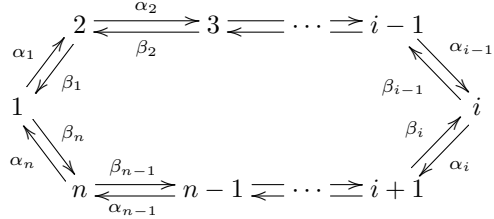
\[
 \xymatrix@C=1.pc@R=1.5pc{
  & 2 \ar@<+.5ex>[rr]^{\alpha_2} \ar@<+.5ex>[ld]^{\beta_1} &&
    3 \ar@<+.5ex>@{-}[r] \ar@<+.5ex>[ll]^{\beta_2} & 
    \cdots \ar@<+.5ex>[r] \ar@<+.5ex>[l] &
    i-1 \ar@<+.5ex>[rd]^{\alpha_{i-1}} \ar@<+.5ex>@{-}[l] \\
  1 \ar@<+.5ex>[ru]^{\alpha_1} \ar@<+.5ex>[rd]^{\beta_n} && && &&
     i \ar@<+.5ex>[ld]^{\alpha_i} \ar@<+.5ex>[lu]^{\beta_{i-1}} \\
  & n \ar@<+.5ex>[lu]^{\alpha_n} \ar@<+.5ex>[rr]^(.4){\beta_{n-1}} &&
    n-1 \ar@<+.5ex>[ll]^(.6){\alpha_{n-1}} \ar@<+.5ex>@{-}[r] &
    \cdots \ar@<+.5ex>[l] \ar@<+.5ex>[r] &
    i+1 \ar@<+.5ex>@{-}[l] \ar@<+.5ex>[ru]^{\beta_i} 
 }
\]
(with $i$ replacing $a_i$),
and the relations
\begin{align*}
&&
 \alpha_i \beta_i &= 0, &
 \beta_i \alpha_i &= 0, &
 (\alpha_i \alpha_{i+1} \dots \alpha_{i-1})^p 
   &= (\beta_{i-1} \beta_{i-2} \dots \beta_{i})^q, 
&&
\end{align*}
for $i \in \{1,\dots,n\}$,
where $\alpha_0  = \alpha_n$,
$\beta_0  = \beta_n$.
\end{example}

\section{Algebras of strict dihedral type}\label{sec:strictdih}

Algebras of dihedral type 
were introduced and studied in \cite{E3} and \cite{E5}. We introduce the
following refinement. 

Assume $A$ is an algebra. We say that $A$  is  of 
\emph{strict dihedral type}
if it satisfies the following conditions:
\begin{enumerate}[(1)]
 \item
  $A$ is symmetric, indecomposable, and tame.
 \item
  The stable Auslander-Reiten quiver $\Gamma_A^s$ of $A$
  consists of the following components:
  \begin{enumerate}[(i)]
   \item
    stable tubes of ranks $1$ and $3$;
   \item
    non-periodic components of the form
    $\bZ \bA_{\infty}^{\infty}$
    or $\bZ \widetilde{\bA}_n$.
   We assume that there is at least one such component.
  \end{enumerate}
 \item
  $\Omega_A$ fixes all stable tubes of rank $3$ in $\Gamma_A^s$.
 \item
  The number $\ell(A)$ of isomorphism classes of simple
  modules in $\mod A$ is two or three.
 \item
  The number of stable tubes of rank $3$ in $\Gamma_A^s$
  is equal to $\ell(A) - 1$.
 \item
  The Cartan matrix $C_A$ of $A$ is non-singular.
\end{enumerate}

An algebra $A$ with $\ell(A)> 1$ 
is of dihedral type if and only if it satisfies conditions (1) to (5).
Then we have the following consequence of
\cite[Theorem~1.4]{E3}.

\begin{theorem}
\label{th:8.1}
Let $K$ be an algebraically closed field of characteristic $2$,
$G$ a finite group, and $B$ be a block of $K G$
whose defect groups are dihedral $2$-groups.
Then $B$ is an algebra of strict dihedral type.
\end{theorem}

We refer to \cite[Tables]{E5} for a complete description
of algebras of dihedral type,  by quivers and relations.

In the rest of this section we will exhibit presentations
of all algebras of strict dihedral type as biserial weighted
triangulation algebras and their socle deformations.
Note that a loop of the triangulation quiver need not be a loop in 
the Gabriel quiver, so that the triangulation quiver allows us 
to unify the description of these algebras. 

\begin{example}
\label{ex:8.2}
Let $(Q,f)$ be the triangulation quiver
\[
  \xymatrix{
    1
    \ar@(ld,ul)^{\alpha}[]
    \ar@<.5ex>,[r]^{\beta}
    & 2
    \ar@<.5ex>[l]^{\gamma}
    \ar@(ru,dr)^{\eta}[]
  } 
\]
with
$f(\alpha) = \beta$,
$f(\beta) = \gamma$,
$f(\gamma) = \alpha$,
and
$f(\eta) = \eta$.
Then we have
$g(\alpha) = \alpha$,
$g(\beta) = \eta$,
$g(\eta) = \gamma$,
and
$g(\gamma) = \beta$.
Hence, $\cO(g)$ consists of the two $g$-orbits $\cO(\alpha)$ and $\cO(\beta)$.
Let $m_{\bullet} : \cO(g) \to \bN^*$
be a weight function, and let
$r = m_{\alpha}$
and
$s = m_{\beta}$.
The border $\partial(Q,f)$ of $(Q,f)$ consists of the vertex $2$
and $\eta$ is the unique border loop.
Take a border function $b_{\bullet} : \partial(Q,f) \to K$,
and set $b = b_2$.
Then the associated
socle deformed biserial
weighted triangulation algebra
\[
    \Lambda(r,s,b) = B(Q,f,m_{\bullet},b_{\bullet})
\]
is given by the above quiver and the relations
\begin{gather*}
 \alpha \beta = 0,
 \ 
 \beta \gamma = 0,
 \ 
 \gamma \alpha = 0,
 \ 
 \alpha^r = (\beta \eta \gamma)^s,
 \ 
 (\gamma \beta \eta)^s = (\eta \gamma \beta)^s, 
 \ 
 \eta^2 = b (\eta \gamma \beta)^s.
\end{gather*}
Clearly, if $b = 0$, then 
$\Lambda(r,s,0) = B(Q,f,m_{\bullet})$.
We note that the family 
$\Lambda(r,s,b)$, $r,s \in \bN^*$, $b \in K$,
coincides with the families
$D(2\cA)$
and
$D(2\cB)$
from \cite[Tables]{E5}.
Finally, we note that 
the considered triangulation quiver $(Q,f)$
is the triangulation quiver $(Q(S,\vv{T}),f)$
of the self-folded triangle $S = T$ of the form
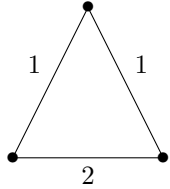
\[
\begin{tikzpicture}[auto]
\coordinate (a) at (0,2);
\coordinate (b) at (-1,0);
\coordinate (c) at (1,0);
\draw (a) to node {1} (c)
(c) to node {2} (b);
\draw (b) to node {1} (a);
\node (a) at (0,2) {$\bullet$};
\node (b) at (-1,0) {$\bullet$};
\node (c) at (1,0) {$\bullet$};
\end{tikzpicture}
\]
with $2$ being the boundary edge and $\vv{T} = (1 \ 1 \ 2)$.
\end{example}

\begin{example}
\label{ex:8.3}
Let $(Q,f)$ be the triangulation quiver
\[
  \xymatrix{
    1
    \ar@(ld,ul)^{\alpha}[]
    \ar@<.5ex>[r]^{\beta}
    & 2
    \ar@<.5ex>[l]^{\gamma}
    \ar@<.5ex>[r]^{\delta}
    & 3
    \ar@<.5ex>[l]^{\eta}
    \ar@(ru,dr)^{\xi}[]
  }
\]
with
$f(\alpha) = \beta$,
$f(\beta) = \gamma$,
$f(\gamma) = \alpha$,
$f(\xi) = \eta$,
$f(\eta) = \delta$,
$f(\delta) = \xi$.
Then we have
$g(\alpha) = \alpha$,
$g(\xi) = \xi$,
$g(\beta) = \delta$,
$g(\delta) = \eta$,
$g(\eta) = \gamma$,
and
$g(\gamma) = \beta$.
Hence $\cO(g)$ consists
of the three $g$-orbits
$\cO(\alpha)$,
$\cO(\beta)$,
and
$\cO(\xi)$.
Let $m_{\bullet} : \cO(g) \to \bN^*$
be a weight function, and
$r = m_{\alpha}$,
$s = m_{\beta}$,
$t = m_{\xi}$.
Observe that the border $\partial(Q,f)$ of $(Q,f)$ is empty.
The associated
biserial weighted triangulation algebra
\[
    \Gamma(r,s,t) = B(Q,f,m_{\bullet})
\]
is given by the above quiver and the relations
\begin{gather*}
 \alpha \beta = 0,
 \qquad
 \beta \gamma = 0,
 \qquad
 \gamma \alpha = 0,
 \qquad
 \delta \xi = 0,
 \qquad
 \xi \eta = 0,
 \qquad
 \eta \delta = 0,
 \\
 \alpha^r = (\beta \delta \eta \gamma)^s,
 \ \ \quad\qquad
 (\gamma \beta \delta \eta)^s = (\delta \eta \gamma \beta)^s, 
 \ \ \quad\qquad
 \xi^t = (\eta \gamma \beta \delta)^s.
\end{gather*}
We note that the family 
$\Gamma(r,s,t)$, $r,s,t \in \bN^*$,
coincides with the families
$D(3\cA)_1$,
$D(3\cB)_1$,
$D(3\calD)_1$
from \cite[Tables]{E5}.
Finally, we observe that 
the considered triangulation quiver $(Q,f)$
is the triangulation quiver $(Q(S,\vv{T}),f)$
for the triangulation $T$
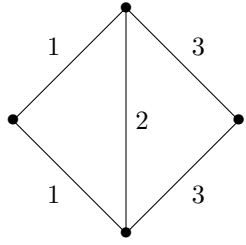
\[
\begin{tikzpicture}[auto]
\coordinate (a) at (0,1.5);
\coordinate (b) at (-1.5,0);
\coordinate (c) at (1.5,0);
\coordinate (d) at (0,-1.5);
\draw (a) to node {3} (c)
(c) to node {3} (d)
(d) to node {1} (b)
(b) to node {1} (a);
\draw (a) to node {2} (d);
\node (a) at (a) {$\bullet$};
\node (b) at (b) {$\bullet$};
\node (c) at (c) {$\bullet$};
\node (d) at (d) {$\bullet$};
\end{tikzpicture}
\]
of the surface $S = \bP \# \bP$
and $\vv{T} = (1 \ 1 \ 2) (3 \ 3 \ 2)$. 
\end{example}

\begin{example}
\label{ex:8.4}
Let $(Q,f)$ be the triangulation quiver
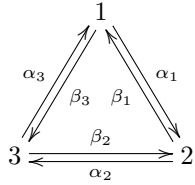
\[
 \xymatrix@R=3.5pc@C=1.8pc{
  &  1
    \ar@<.35ex>[rd]^{\alpha_1}
    \ar@<.35ex>[ld]^{\beta_3} \\
    3
    \ar@<.35ex>[ru]^{\alpha_3}
    \ar@<.35ex>[rr]^{\beta_2}
    && 2
    \ar@<.35ex>[lu]^{\beta_1}
    \ar@<.35ex>[ll]^{\alpha_2}
    \\
  }
\]
with
$f(\alpha_1) = \alpha_2$,
$f(\alpha_2) = \alpha_3$,
$f(\alpha_3) = \alpha_1$,
$f(\beta_1) = \beta_3$,
$f(\beta_3) = \beta_2$,
$f(\beta_2) = \beta_1$.
Then we have
$g(\alpha_1) = \beta_1$,
$g(\beta_1) = \alpha_1$,
$g(\alpha_2) = \beta_2$,
$g(\beta_2) = \alpha_2$,
$g(\alpha_3) = \beta_3$,
and
$g(\beta_3) = \alpha_3$.
Hence $\cO(g)$ consists
of the three $g$-orbits
$\cO(\alpha_1)$,
$\cO(\alpha_2)$,
$\cO(\alpha_3)$.
Let $m_{\bullet} : \cO(g) \to \bN^*$
be a weight function, and
$m_1 = m_{\alpha_1}$,
$m_2 = m_{\alpha_2}$,
$m_3 = m_{\alpha_3}$.
Then the associated
biserial weighted triangulation algebra
\[
    \Omega(m_1,m_2,m_3) = B(Q,f,m_{\bullet})
\]
is given by the above quiver and the relations
\begin{gather*}
 \alpha_1 \alpha_2 = 0,
 \quad
 \alpha_2 \alpha_3 = 0,
 \quad
 \alpha_3 \alpha_1 = 0,
 \quad
 \beta_1 \beta_3 = 0,
 \quad
 \beta_3 \beta_2 = 0,
 \quad
 \beta_2 \beta_1 = 0,
 \\
 (\alpha_1 \beta_1)^{m_1} = (\beta_3 \alpha_3)^{m_3},
 \quad
 (\beta_1 \alpha_1)^{m_1} = (\alpha_2 \beta_2)^{m_2},
 \quad
 (\beta_2 \alpha_2)^{m_2} = (\alpha_3 \beta_3)^{m_3}.
\end{gather*}
We note that the family 
$\Omega(m_1,m_2,m_3)$, $m_1,m_2,m_3 \in \bN^*$,
coincides with the family
$D(3\cK)$
from \cite[Tables]{E5}.
Finally, 
the considered triangulation quiver $(Q,f)$
is the triangulation quiver $(Q(\bS,\vv{T}),f)$
associated to the following triangulation $T$ of the sphere $\bS$
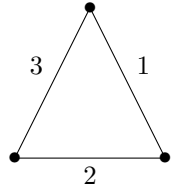
\[
\begin{tikzpicture}[auto]
\coordinate (a) at (0,2);
\coordinate (b) at (-1,0);
\coordinate (c) at (1,0);
\draw (a) to node {1} (c)
(c) to node {2} (b);
\draw (b) to node {3} (a);
\node (a) at (0,2) {$\bullet$};
\node (b) at (-1,0) {$\bullet$};
\node (c) at (1,0) {$\bullet$};
\end{tikzpicture}
\]
and the orientation $\vv{T}$ of two triangles in $T$
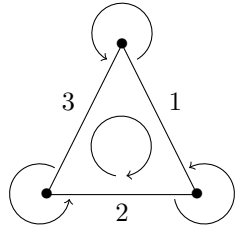
\[
\begin{tikzpicture}[auto]
\coordinate (a) at (0,2);
\coordinate (b) at (-1,0);
\coordinate (c) at (1,0);
\coordinate (d) at (-.08,.25);
\coordinate (aa) at (-.2,1.8);
\coordinate (bb) at (-.7,-.06);
\coordinate (cc) at (.7,-.06);
\draw (a) to node {1} (c)
(c) to node {2} (b)
(b) to node {3} (a);
\draw[->] (d) arc (260:-80:.4);
\node (a) at (0,2) {$\bullet$};
\node (b) at (-1,0) {$\bullet$};
\node (c) at (1,0) {$\bullet$};
\draw[<-] (aa) arc (240:-60:.4);
\draw[<-] (bb) arc (350:60:.4);
\draw[->] (cc) arc (-170:120:.4);
\end{tikzpicture}
\]
\end{example}

The following theorem is a complete classification 
of the basic algebras of 
strict dihedral type,
established in \cite{E3} and \cite{E5}.

\begin{theorem}
\label{th:8.5}
Let $A$ be a basic algebra over an algebraically
closed field $K$, with the Grothendieck group $K_0(A)$
of rank at least $2$.
The following statements are equivalent:
\begin{enumerate}[(i)]
  \item
    $A$ is an algebra of strict dihedral type.
  \item
    $A$ is isomorphic to an algebra of one of the forms
    $\Lambda(r,s,b)$,
    $\Gamma(r,s,t)$,
    or $\Omega(m_1,m_2,m_3)$.
\end{enumerate}
\end{theorem}

Moreover, we have the following consequence of the
results established in \cite[Section~4]{ESk1}.

\begin{corollary}
\label{cor:8.6}
Let $A$ be a basic algebra of strict dihedral type over an algebraically
closed field $K$. 
Then the following statements are equivalent:
\begin{enumerate}[(i)]
  \item
    $A$ is of polynomial growth.
  \item
    $A$ is isomorphic to an algebra of one of the forms
    $\Lambda(1,1,0)$,
    $\Lambda(1,1,1)$,
    $\Gamma(1,1,1)$,
    or $\Omega(1,1,1)$.
\end{enumerate}
\end{corollary}

\section{Proof of Theorem \ref{th:main1}}\label{sec:proof1}

The aim of this section is to prove  Theorem \ref{th:main1}.
The implication (ii) $\Rightarrow$ (i) follows from
Propositions 
\ref{prop:5.1},
\ref{prop:6.2},
\ref{prop:6.3},
\ref{prop:6.4},
\ref{prop:7.1},
\ref{prop:7.2},
\ref{prop:7.3},
noting that $(Q(S,\vv{T}),f)$ is a triangulation quiver.
Hence 
it remains to show that (i) implies (ii).
We split the proof into several steps.

\bigskip

\begin{lemma}
\label{lem:9.1}
Let $A = KQ/I$ be a symmetric bound quiver algebra,
and $i,j$ two different vertices in $Q$.
Assume that the simple module $S_i$ in $\mod A$ 
is periodic of period $3$.
Then the following statements hold.
\begin{enumerate}[(i)]
  \item
    The number of arrows in $Q$ from $i$ to $j$
    is the same as the number of arrows in $Q$
    from $j$ to $i$.
  \item
    No path of length $2$ between $i$ to $j$
    occurs in a minimal relation of $I$.
\end{enumerate}
\end{lemma}

\begin{proof}
This follows from
\cite[Lemmas IV.1.5 and IV.1.7]{E5}. 
\end{proof}

\begin{lemma}
\label{lem:9.2}
Let $A = KQ/I$ be an algebra of generalized dihedral type
with $|Q_0| \geq 3$,
and $i \in Q_0$ a vertex such that $|s^{-1}(i)| \geq 2$
and $S_i$ is a periodic module in $\mod A$.
Then, for any vertex $j$ in $Q$ connected to $i$ by an arrow,
the simple module $S_j$ is a periodic module in $\mod A$.  
\end{lemma}

\begin{proof} By Lemma \ref{lem:9.1} we know $m:= |s^{-1}(i)|=|t^{-1}(i)|$, and we assume $m\geq 2$.  Let $\alpha_1, \ldots, \alpha_m$ be the
arrows in $Q$ starting at $i$, and let $j_r:= t(\alpha_r)$.

Since $S_i$ has period $3$,
there exists an exact sequence in $\mod A$
\[
  0 \rightarrow
  S_i \xrightarrow{d_3}
  P_i \xrightarrow{d_2}
  \bigoplus_{r=1}^m P_{j_r} \xrightarrow{d_1}
  P_i \xrightarrow{d_0}
  S_i \rightarrow
  0
\]
with $d_0$ the canonical epimorphism, $d_3$ the canonical monomorphisms, and where $d_1$ is defined as
$$d_1(x_1, \ldots, x_m):= \sum_{r=1}^m \alpha_ix_i
$$
for any $(x_1, \ldots, x_m) \in \bigoplus_{j=1}^r P_{j_r}$. 

Since $\bigoplus_{r=1}^m P_{j_r}$ is the injective hull of 
$\Ker d_1 \cong \Omega^2(S_i) \cong \Omega^{-1}(S_i)$, 
we deduce that the $m$ arrows ending at $i$ have starting
vertices $j_1, \ldots, j_m$.  

(1) We claim that the $j_r$ are pairwise distinct: If not, say $j_2=j_1$, then we have double arrows from $i$ to $j_1$ and also from $j_1$ to $i$. By assumption,
$|Q_0|\geq 3$ and $Q$ is connected, and therefore there is a vertex $k\neq i, j_1$ and some arrow between $k$ and one of $i$ or $j_1$. Then $A$ has a quotient algebra
$K\Delta$ where $\Delta$ is the  wild quiver of the form  
\[
  \xymatrix{
    \bullet
    & \bullet
    \ar@<.5ex>[l]
    \ar@<-.5ex>[l]
    \ar[r]
    & \bullet
}
,
\]
or its opposite quiver, which is a contradiction since $A$ is tame.

(2) We 
have $\Ker d_1 \cong \Omega^{-1}(S_i)\cong P_i/S_i$, and this is a cyclic $A$-module.
So it is of the form $(\beta_1, \ldots, \beta_m)A$ with $\beta_r\in P_{j_r}$. We know that the $j_r$ are distinct, that the arrows ending at
$i$ have starting vertices $j_r$, and that $P_i/S_i$ is isomorphic to $(\beta_1, \ldots, \beta_m)A$ with  $\beta_r$  in $P_{j_r}$. This implies that $\beta_r$ is not in 
$\rad^2(P_{j_r})$, so we can
take the $\beta_r$ as the arrows ending at $i$. 
Now we have
$$\sum_{r=1}^m \alpha_rx_r=0 \ \ \Leftrightarrow \ \ (x_1, \ldots, x_r) = (\beta_1, \ldots, \beta_m)a  \ (\mbox{some } a\in A).
$$
Suppose now that  one of the simple modules
$S_{j_1}, \dots, S_{j_m}$, say $S_{j_1}$,  
is not periodic.
We set $j = j_1$, $\alpha = \alpha_1$, and $\beta = \beta_1$.

(3) There is an exact sequence in $\mod A$
\[
  0 \rightarrow
  \Omega_A(\alpha A) \rightarrow
  P_j \xrightarrow{d}
  \alpha A \rightarrow
  0
\]
with 
$d(x) = \alpha x$ for $x \in P_j$,
and $\Omega_A(\alpha A) = \{ x \in P_j \,|\, \alpha x = 0\}$.
If $x\in \Omega_A(\alpha A)$ then $(x, 0, \ldots, 0)\in \Ker d_1$ and hence $(x, 0, \ldots, 0)= (\beta, \beta_2, \ldots, \beta_m)a$ for $a\in A$
so that $x=\beta a$ and therefore
$\Omega_A(\alpha A) \subseteq \beta A$.

\medskip

We will now show that $\Omega_A(\alpha A) = \beta A$. 
Note first that $\rad P_j/\Omega_A(\alpha A)$ has a simple socle (since $\alpha A$ has a simple socle).
Since $S_j$ is not periodic, the quotient $\rad P_j/S_j$  
is a direct sum of two non-zero modules.
That is,  there are
submodules $M$ and $N$ of $\rad P_j$ such that
$\rad P_j = M + N$ and $M \cap N = S_j$.

We know (by (1)) that $S_i$ is a direct summand of $\rad P_j/\rad^2P_j$ with multiplicity one, so
we may assume that $S_i$ is a direct summand
of $M / \rad M$.
Now $\beta = e_j\beta e_i\not\in\rad^2 P_j$ and hence $\beta A\subseteq M$. 

We have now $S_j \subseteq \Omega_A(\alpha A)\subseteq M \subseteq \rad P_j$
and since 
$\rad P_j/S_j = (M/S_j) \oplus (N/S_j)$,
we conclude that 
$\rad P_j/\Omega_A(\alpha A) = M/\Omega_A(\alpha A) \oplus N/S_j$. 
Now $N/S_j\neq 0$ and $\rad P_j/\Omega_A(\alpha A)$ has a simple socle, so it
follows that $M/\Omega_A(\alpha A) = 0$ and therefore 
$\Omega_A(\alpha A) = \beta A = M$. 

Now 
$\beta A = \beta A \oplus 0 \oplus \dots \oplus 0 \oplus 0$
is contained in $(\beta, \beta_2, \dots, \beta_m) A$.
In particular, there is $y \in A$ such that $\beta = \beta y$
and $\beta_r y = 0$ for any $r \in \{2,\dots, m\}$. 
Let $z = e_i y e_i$, then 
$\beta = \beta z$ and $\beta_r z = 0$.
Since $\beta_2 \neq 0$, we see that $z$ is a non-invertible element
of the local algebra $e_i A e_i$.
But then ${e_i}-z$ is an invertible element of
$e_i A e_i$ and $\beta({e_i}-z) = 0$,
which contradicts the fact that  $\beta \neq 0$.
Hence $S_{j_1}$ and similarly 
 all simple modules
$S_{j_2},\dots,S_{j_m}$ are periodic.
\end{proof}

\begin{proposition}
\label{prop:9.3}
Let $A = KQ/I$ be an algebra of generalized dihedral type,
with $|Q_0| \geq 3$,
and let $i$ be a vertex of $Q$.
Then the following statements are equivalent:
\begin{enumerate}[(i)]
 \item
  $S_i$ is a periodic module in $\mod A$.
 \item
  $|s^{-1}(i)| = 1$.
 \item
  $|t^{-1}(i)| = 1$.
 \item
  There are unique arrows 
  $
    \xymatrix{ i \ar@<.5ex>[r]^{\beta} & j \ar@<.5ex>[l]^{\alpha} }
  $
  adjacent to  the vertex $i$, and $\beta \alpha = 0$.
\end{enumerate}
\end{proposition}

\begin{proof}
First we show that for any simple periodic module $S_i$ 
we have 
$|s^{-1}(i)| = |t^{-1}(i)|=1$.
By definition of an algebra of generalized dihedral type
and Corollary~\ref{cor:3.2},
$\mod A$ admits a non-periodic simple module $S_k$.
If $t$ is a vertex connected to $k$ by an arrow and if $S_t$ is periodic then $S_t$ does not satisfy
the assumption of 
Lemma \ref{lem:9.2}, therefore   $|s^{-1}(t)|=|t^{-1}(t)|=1$. If $S_t$ is not periodic we consider
similarly all vertices $j$ which are connected to $t$ by an arrow. If $S_j$ is periodic then 
we apply Lemma \ref{lem:9.2}, else we repeat the process. 
The quiver is connected and after finitely many steps we have reached all vertices, and the claim follows.
Hence (i) implies (ii) and (iii).

Assume now that $|s^{-1}(i)| = 1$.
Let $\beta$ be the arrow in $Q$ with source $i$ and $j = t(\beta)$.
Then $\rad P_i = \beta A$ which has simple top, and
hence the quotient $\rad P_i / S_i$ 
is indecomposable.
But then $S_i$ is a periodic module of period $3$, because
$A$ is of generalized dihedral type.
Consider an exact sequence in $\mod A$
\[
  0 \rightarrow
  S_i \xrightarrow{d_3}
  P_i \xrightarrow{d_2}
  P_{j} \xrightarrow{d_1}
  P_i \xrightarrow{d_0}
  S_i \rightarrow
  0
\]
which gives rise to a minimal projective resolution of $S_i$
in $\mod A$.
Then
$P_i/S_i = \Omega_A^{-1}(S_i) = \Omega_A^{2}(S_i)$
is a submodule of $P_j$ of the form $\alpha A$ 
for a unique arrow $\alpha$ in $Q$ from $j$ to $i$.
Therefore, we have in $Q$ only two arrows
\[
    \xymatrix{ i \ar@<.5ex>[r]^{\beta} & j \ar@<.5ex>[l]^{\alpha} }
\]
containing the vertex $i$.
Moreover, from the proof of
Lemma~\ref{lem:9.2},
we have $\beta\alpha=0$.
Hence, (ii) implies (i) and (iv).
Similarly, (iii) implies (i) and (iv).
Obviously, (iv) implies (ii) and (iii).
This finishes the proof.
\end{proof}

\begin{lemma}
\label{lem:9.4}
Let $A = KQ/I$ be a bound quiver algebra of generalized dihedral type
with $|Q_0| \geq 3$,
$i$ a vertex of $Q$ with $|s^{-1}(i)| = 1 = |t^{-1}(i)|$,
and
\[
    \xymatrix{ i \ar@<.5ex>[r]^{\beta} & j \ar@<.5ex>[l]^{\alpha} }
\]
be the unique arrows in $Q$ adjacent to  $i$.
Then $\alpha \beta$ does not occur in a minimal relation of $A$.
\end{lemma}

\begin{proof}
Assume that $\alpha \beta$ occurs in a minimal relation of $A$.
Since $\alpha \notin \soc(A)$, we have $\alpha(\rad A) \neq 0$,
and consequently $\alpha(\rad A)$ is not contained in $\alpha(\rad A)^2$.
Observe also that $\alpha(\rad A) = \alpha \beta A$.
Then $Q$ contains a subquiver
\[
    \xymatrix{ 
      i \ar@<.5ex>[r]^{\beta} & j \ar@<.5ex>[l]^{\alpha} 
        \ar@<.5ex>[r]^{\gamma} & k \ar@<.5ex>[l]^{\delta} 
    }
\]
such that $\alpha \beta$ and $\gamma \delta$ occur 
in the same minimal relation of $A$.
It follows from 
Proposition~\ref{prop:9.3}
that $S_i$ is a periodic module of period $3$ and
$\beta \alpha = 0$
(see also the proof of Lemma~\ref{lem:9.2}).
We also note that $S_j$ is not periodic because
$|s^{-1}(j)| \geq 2$.
This implies that the heart $\rad P_j/S_j$ is a direct sum
of two non-zero submodules.
We have two cases to consider.

\smallskip

(1)
Assume that $|s^{-1}(j)| = 2 = |t^{-1}(j)|$.
Then we have $\rad P_j = \alpha A + \gamma A$
and $\alpha A \cap \gamma A = S_j$.
On the other hand, we have a minimal relation in $A$ of the form
\[
  \alpha \beta + \alpha x + a \gamma \delta + \gamma y = 0
\]
for some $a \in K^*$ and $x \in e_i(\rad A)^2$, $y \in e_k(\rad A)^2$.
Moreover, $x = \beta z$ for some $z \in e_j (\rad A) e_j$.
Hence, we get
\[
  \alpha \beta (1 - z)+ \gamma (a \delta + y) = 0 .
\] 
Since $1 - z$ is an invertible element of $A$, we obtain
an equality
\[
  \alpha \beta = \gamma \delta' ,
\] 
where $\delta' \in e_k(\rad A) e_j \setminus e_k(\rad A)^2 e_j$.
This implies that $\alpha \beta$ generates $\soc(P_j)$,
because 
$\alpha A \cap \gamma A = S_j = \soc (P_j)$.
Let $\varphi : A \to K$ be a symmetrizing $K$-linear form.
Then 
$\varphi(\alpha \beta) = \varphi(\beta \alpha) = \varphi(0) = 0$.
Thus $K \alpha \beta$ is a non-zero ideal in $A$
which is contained in $\Ker \varphi$, a contradiction.
 
\smallskip

(2)
Assume that $|s^{-1}(j)| \geq 3$.
Then there is an arrow $\sigma$ in $Q$ with $s(\sigma) = j$,
different from $\alpha$ and $\gamma$.
We set $l = t(\sigma)$.
Observe that $l \notin \{i,k\}$,
because otherwise $A$ admits a quotient algebra isomorphic
to the path algebra of the wild quiver of the form
\[
  \xymatrix{
    \bullet
    & \bullet
    \ar@<.5ex>[l]
    \ar@<-.5ex>[l]
    \ar[r]
    & \bullet
  }
 .
\]
Recall that $S_i$ is a periodic module of period $3$.
Then it follows from 
Lemma~\ref{lem:9.1}
that the paths
$\beta \gamma$,
$\beta \sigma$,
$\delta \alpha$
do not occur in minimal relations of $A$,
and
$\beta \gamma$,
$\beta \sigma$,
$\delta \alpha$
are non-zero paths.
We also note that
$\delta \gamma \neq 0$ 
because $\delta \alpha \beta \neq 0$.
Similarly, $\alpha \beta \sigma \neq 0$
forces $\delta \sigma \neq 0$.
Let $B$ be the quotient algebra of $A$ 
by the ideal generated by all paths of length $2$
except
$\beta \gamma$,
$\beta \sigma$,
$\delta \alpha$,
$\delta \gamma$,
$\delta \sigma$.
Then there is a Galois covering 
$F : R \to R/G = B$, with a finitely generated free group $G$,
such that the locally bounded $K$-category $R$ admits a full
convex subcategory $\Lambda$ isomorphic to the bound quiver
algebra $C = K \Delta / L$, where $\Delta$ is the quiver
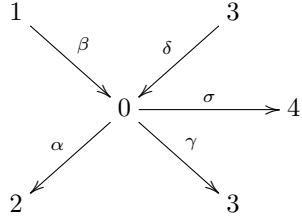
\[
  \xymatrix@C=1pc{
   1 \ar[rrd]^{\beta} && && 3 \ar[lld]_{\delta} \\
    && 0 \ar[lld]_{\alpha} 
     \ar[rrr]^{\sigma} \ar[rrd]^{\gamma} &&& 4
     \\
    2 && && 3 
  }
\]
and $L$ is the ideal in $K \Delta$ generated by $\beta \alpha$.
Then $C$ is a wild concealed algebra of the tree type
$\tilde{\tilde{\bD}}_4$
\[
  \xymatrix{
    &&& \bullet \\
    \bullet \ar@{-}[r] & \bullet \ar@{-}[r] 
      & \bullet \ar@{-}[ru] \ar@{-}[r] \ar@{-}[rd]
     & \bullet \\
    &&& \bullet 
  }
    \xymatrix@R=2.35pc{ \\ \\ .}
\]
Applying 
\cite[Proposition~2]{DS1}
and
\cite[Theorem]{DS2}
we conclude that $B$ is a wild algebra.
This is a contradiction, because $B$ is a quotient algebra
of the tame algebra $A$.
\end{proof}

\begin{proposition}
\label{prop:9.5}
Let $A = KQ/I$ be a bound quiver algebra of generalized dihedral type
with $|Q_0| \geq 3$.
Then the following conditions are satisfied.
\begin{enumerate}[(Q1)]
 \item
  For each vertex $i \in Q_0$, we have
  $|s^{-1}(i)| = |t^{-1}(i)|$ and this $1$ or $2$.
 \item
  For each vertex $i \in Q_0$ with
  $|s^{-1}(i)| = 2$, $\rad P_i/S_i = U \oplus V$
  with $U, V$ indecomposable modules.
\end{enumerate}
\end{proposition}

\begin{proof}
Let $I_0 = \{ i \in Q_0 \,|\, |s^{-1}(i)| = 1\}$.
Consider the idempotent $\varepsilon = 1_A - \sum_{i \in I_0} e_i$
and the associated idempotent algebra $B = \varepsilon A \varepsilon$.
We note that $B$ is a tame algebra
(see \cite[Theorem]{DS2}),
because $A$ is a tame algebra.
Clearly, $B = A$ if $I_0$ is empty.
For each vertex $i \in I_0$, we have in $Q_0$ unique arrows
\[
    \xymatrix{ i \ar@<.5ex>[r]^{\beta_i} & i^* \ar@<.5ex>[l]^{\alpha_i} }
\]
containing the vertex $i$.
Moreover, by 
Proposition~\ref{prop:9.3},
the simple module $S_i$ is periodic.
Since $|Q_0| \geq 3$,
applying
Proposition~\ref{prop:9.3}
again,
we conclude that $|s^{-1}(i^*)| \geq 2$
and $|t^{-1}(i^*)| \geq 2$.
It follows also from
Lemma~\ref{lem:9.1}
that $\alpha_i \beta_i$ does not occur in a minimal relation of $A$.
Therefore, $B$ is the bound quiver algebra $K Q^* / I^*$,
where $Q^*$ is the quiver obtained from $Q$ by replacing all
$2$-cycles 
$\xymatrix{ i \ar@<.5ex>[r]^{\beta_i} & i^* \ar@<.5ex>[l]^{\alpha_i} }$,
for $i \in I_0$, by the loops $\eta_i$ at $i^*$ and keeping all
other arrows of $Q$, and $I^* = \varepsilon I \varepsilon$.
We note that for any vertex $k$ of $Q^*$ we have
$|s^{-1}(k)| \geq 2$ and $|t^{-1}(k)| \geq 2$.
We claim that
$|s^{-1}(k)| = 2 = |t^{-1}(k)|$
for any vertex $k$ of $Q^*$.
Suppose that 
$|s^{-1}(j)| \geq 3$ or $|t^{-1}(j)| \geq 3$
for a vertex $j$ of $Q^*$.
Without loss of generality, we may assume that $|s^{-1}(j)| \geq 3$.
Consider the quotient algebra $D = B/ (\rad B)^2$.
Then $D$ admits a Galois covering 
$F : R \to R/G = D$, with a finitely generated free group $G$,
such that the locally bounded $K$-category $R$ contains a full
convex subcategory $\Lambda$ isomorphic to the path
algebra $K \Delta$ 
of the wild quiver $\Delta$ 
of type $\tilde{\tilde{\bE}}_7$
\[
  \xymatrix@C=1.8pc{
     & \bullet \ar[ld] \ar[rd]
     && \bullet \ar[ld] \ar[d] \ar[rd]
     && \bullet \ar[ld] \ar[rd]
     && \bullet \ar[ld]
   \\
     \bullet && \bullet & \bullet & \bullet && \bullet
  }
  \xymatrix@R=2.35pc{ \\  .}
\]
Applying 
\cite[Proposition~2]{DS1}
and
\cite[Theorem]{DS2}
again,
we conclude that $B$ is a wild algebra, a contradiction.
Therefore, indeed 
$|s^{-1}(k)| = 2 = |t^{-1}(k)|$
for any vertex $k \in Q_0^*$.

This shows that $A$ satisfies (Q1).
But then it follows from
Proposition~\ref{prop:9.3}
and Remark~\ref{rem:3.3}
that $A$ satisfies also (Q2).
\end{proof}

\begin{lemma}\label{lem:9.6}The statements of Propositions \ref{prop:9.5} and 
\ref{prop:9.3} also hold when $|Q_0|=2$.
\end{lemma}


\begin{proof} 
Suppose $Q_0=\{ 1, 2\}$. 
It follows from Corollary~\ref{cor:3.2} that
at least one of the simple modules,
say $S_1$, is non-periodic, and then $(\rad P_1)/S_1 = U_1\oplus V_1$ with 
two non-zero summands. Hence there are at least two arrows starting at
$1$ and two arrows ending at $1$. That is, $Q$ contains one of 
the following quivers
\begin{enumerate}[(a)]
 \item
  \raisebox{.5ex}%
  {\xymatrix{ 1 \ar@(dl,ul)[]^{\alpha} \ar@/^1.5ex/[r]^{\beta} & 2 \ar@/^1.5ex/[l]^{\gamma}}}.
 \item
  \raisebox{.5ex}%
  {\xymatrix{ 
     1 \ar@/^1.95ex/[r]_{\alpha_2} 
        \ar@<+.95ex>@/^1.95ex/[r]^{\alpha_1} & 
     2 \ar@/^1.95ex/[l]_{\beta_2} 
        \ar@<+.95ex>@/^1.95ex/[l]^{\beta_1}
   }}.
This is then all of $Q$ since $A$ is tame.
\end{enumerate}

Note that since $A$ is tame, it cannot have a connected quiver with
two loops at some vertex.

\bigskip

Consider first the quiver (b). 
If also $S_2$ is non-periodic then Proposition~\ref{prop:9.5} 
follows directly, 
and Proposition~\ref{prop:9.3} is vacuously true. 
So assume for a contradiction that $S_2$ is
periodic, and then it has period three. 
We have
an exact sequence
\[
0 \to S_2 \to P_2 \to P_1 \oplus P_1 \to P_2 \to S_2\to 0
\]
and it follows that in the Grothendieck group $K_0(A)$ we have
\[
  [P_1] = [P_2] - [S_2].
\]
We may take the arrows $\alpha_1, \alpha_2$ so that
$\alpha_1A \cap \alpha_2A = S_1$, and then 
$[\alpha_1A] + [\alpha_2A] = [P_1]$. 

As well we know that $(\alpha_1, \alpha_2)A \cong P_2/S_2$ in $\mod A$.
We have an exact
sequence
\[
  0 \to (\alpha_1, \alpha_2)A \to \alpha_1A\oplus \alpha_2A \to C \to 0
  .
\]
The first term has composition factors $[P_2]-[S_2]$ and the second term 
has composition factors $[P_1]$ and these are equal. 
Hence $C=0$ and  $P_2/S_2$ is a direct sum, a contradiction. 
This shows that Proposition~\ref{prop:9.3} holds as well in this case.

\bigskip

Consider now the quiver (a). 
Assume $S_2$  is not periodic.
Then the quiver $Q$ has a loop at both vertices,
and we have  
Propositions \ref{prop:9.5} and \ref{prop:9.3}. 
So assume now that $S_2$ is periodic 
and we know that $|s^{-1}(2)| = |t^{-1}(2)|$ and this is $1$ or $2$. 
We are
done if we show that it is equal to $1$. 
Assume this is false, then we have an exact
sequence
\[
  0 \to S_2 \to P_2 \to P_1\oplus P_2 \to P_2 \to S_2 \to 0 ,
\]
and it follows that $[P_1] = [P_2]-2[S_2]$. 
Therefore the Cartan matrix is non-singular.  
This contradicts \cite[Lemma~VI.1.1 and Theorem~VI.8.2]{E5}.
\end{proof}

From now on until the end of this section we assume that 
$A = KQ/I$ is a bound quiver algebra 
of generalized dihedral type.
Then
$A$ satisfies the conditions (Q1) and (Q2) of Proposition~\ref{prop:9.5}.

\begin{notation}
\label{notation:9.7}
We say that a vertex $i\in Q_0$ is a $1$-vertex if $|s^{-1}(i)| =  1$
and a $2$-vertex otherwise.
We denote by  $I_0$  the set of all $1$-vertices of $Q$.
For any $2$-vertex $i$ there are  arrows $\alpha, \beta$ starting at $i$ such
that
$$\alpha A \cap \beta A = \soc (e_iA)$$
We fix a set of arrows which satisfy this.
\end{notation}

For such a choice of arrows,
if
$\alpha, \beta \in Q_1$ start at $i$ then
$\rad P_i/S_i = U_i \oplus V_i$, where
$U_i = \alpha A/S_i$ and $V_i = \beta A/S_i$.

For an arrow $\gamma \in Q_1$ and $i = t(\gamma)$,
we set
\[
   R_{\gamma} = \big\{ x \in e_i A \,|\, \gamma x = 0 \big\} .
\]
We note that $R_{\gamma}$ is isomorphic to
$\Omega_A(\gamma A)$, and we will always take this as an identification.
The following lemma from 
\cite[Lemma~VI.1.1]{E5}
provides another description of the hearts of
indecomposable projective modules associated to $2$-vertices.

\begin{lemma}
\label{lem:9.8}
Let $i$ be a vertex of $Q$ at which two arrows $\gamma$ and $\delta$
end.
Then 
$\rad P_i = R_{\gamma} + R_{\delta}$,
$R_{\gamma} \cap R_{\delta} = S_i$,
and hence
$(\rad P_i) / S_i = (R_{\gamma}/S_i) \oplus (R_{\delta}/S_i)$.
\end{lemma}

 If
we modify some arrows but keep the intersection
condition of Notation~\ref{notation:9.7} then the collection 
of modified arrows also satisfies
Lemma \ref{lem:9.8}.

As well we have $(\rad P_i) / S_i = \alpha A/S_i \oplus \beta A/S_i$, the direct
sum of two indecomposable modules. The Krull-Schmidt Theorem gives 
that $R_{\gamma}/S_i$ is isomorphic to one of $\alpha A/S_i$ or
$\beta A/S_i$. 


\begin{notation}\label{notation:9.9}\normalfont
(1) \ We define a map $f : Q_1 \cup I_0 \to Q_1 \cup I_0$ by
\[
  f(\gamma) = \left\{
   \begin{array}{cl}
   \alpha & \mbox{if $R_{\gamma}/S_i \cong (\alpha A)/S_i$,} \\
   i & \mbox{if $R_{\gamma} = S_i$},
   \end{array}
  \right.
\]
for an arrow $\gamma \in Q_1$, and for $i \in I_0$, we define
\[
  f(i) = \alpha,
\]
where $\alpha$ is a unique arrow in $Q_1$ with source $i$.

This is a permutation of $Q_1 \cup I_0$. We note that
if no double arrows start at vertex $i$ then 
$R_{\gamma}/S_i \cong (\alpha A)/S_i$ if and only if
$R_{\gamma} = \alpha'A$ where $\alpha - \alpha' \in (\rad A)^2$. 

If there are no double arrows then the map  $f$ is the same as the map
 $\pi$ in 
\cite[VI.1.2]{E5}.

(2) \
There is also the permutation of $Q_1$ which describes
the composition series of modules generated by arrows.
It is called $\pi^*$ in
\cite[VI.1.3]{E5},
and we will denote it by $g$.

We define the permutation $g : Q_1 \to Q_1$ as follows
\[
  g(\gamma) = \left\{
   \begin{array}{cl}
   \alpha & \mbox{if $t(\gamma) \in I_0$ and $\alpha \in Q_1$ with $s(\alpha) = t(\gamma)$},\\
   \delta & \mbox{if $t(\gamma) \notin I_0$ and $\delta \in Q_1 \setminus \{ f(\gamma)\}$ with $s(\delta) = t(\gamma)$}.
   \end{array}
  \right.
\]

(3) \ 
The permutation $f$ describes the action of $\Omega_A$
on modules in $\mod A$ generated by arrows.
In fact, it follows from 
\cite[Theorem~IV.4.2]{E5}
that any Auslander-Reiten sequence in $\mod A$ with the right term
$\alpha A$, for $\alpha \in Q_1$, has indecomposable middle term.
In our setting, these modules occur at mouths of stable tubes, and
stable tubes have rank $1$ or $3$, which means that the modules
generated by arrows are periodic of period at most three 
(with respect to $\Omega_A$).
Furthermore, by Lemma \ref{lem:9.8} we know that for each arrow $\alpha$, 
also $\Omega(\alpha A)$ is generated by an arrow (and hence also 
$\Omega^2(\alpha A)$ in the case when $\alpha A$ has period three). 

\bigskip
We summarize the possibilities
for the cycles of the permutation $f$ of $Q_1 \cup I_0$.
\begin{enumerate}[(i)]
  \item
If $\alpha$ is an arrow occuring in an  $f$-cycle of a vertex in $I_0$ 
that $\alpha A$ has period three.
This follows from Proposition \ref{prop:9.3}. 
\end{enumerate}
Suppose  $\alpha: i\to j$ is an arrow whose $f$-cycle consists of arrows.
\begin{enumerate}[(i)]
\addtocounter{enumi}{1}
  \item
If $\alpha A$ has period one then $\alpha$ is a loop fixed by $f$,

  \item
Suppose $\alpha A$ has period two. Then we
have an exact sequence
$$
0 \to \alpha A \to e_iA \to e_jA \to \alpha A \to 0
$$
and hence $\Omega(\alpha A)$ is generated by an arrow $j\to i$. 
In particular,
if $j\neq i$ this can only occur if there is an arrow $j\to i$. The $f$-cycle
is then either $(\alpha)$ or $(\alpha \ \beta)$, 

  \item
Suppose $\alpha A$ has period three. Then there
is an exact sequence
$$
0 \to \alpha A \to e_iA \to e_kA \to e_jA \to \alpha A \to 0
$$
and $\Omega(\alpha A)$ is generated by an arrow $\delta: j\to k$, and $\Omega^2(\alpha A)$ is generated by an arrow $\gamma: k\to i$. 
If the $f$-cycle contains a loop then the other two arrows
lie on a subquiver 
$\xymatrix{ a \ar@<.5ex>[r] & b \ar@<.5ex>[l] }$.
Otherwise the three arrows form a triangular subquiver with three
different vertices.
\end{enumerate}
\end{notation}


The next lemma is a variation of \cite[Lemma~VI.1.4.4]{E5}. 
Consider a 2-vertex $i$ of $Q$, there are either four  distinct
arrows adjacent to $i$, or else three when one of them is a loop.
This holds since $Q$ is connected, with at least three vertices.

\begin{lemma}
\label{lem:9.10}
Assume $i \in Q_0$ is a vertex at which two arrows
$\alpha$, $\beta$ start and two arrows $\gamma$, $\delta$ end,
and $f(\gamma) = \alpha$ and $f(\delta) = \beta$.
Then the following statements hold.
\begin{enumerate}[(i)]
 \item
  Suppose $\gamma, \delta$ are
not fixed by $f$. Then there are arrows  
  $\alpha'$ and $\beta'$ with 
$\alpha'A = R_{\gamma}$ and $\beta'A=R_{\delta}$, 
 such that 
   $\gamma \alpha' = 0$ and $\delta \beta' = 0$,
  and $\alpha' A \cap \beta' A \subseteq \soc (e_i A)$.
If $\alpha, \beta$ are not double arrows we may assume
 $\alpha-\alpha'\in (\rad A)^2$
and $\beta-\beta' \in (\rad A)^2$.   
 \item
  Suppose $f(\gamma) = \gamma$ so that $\gamma = \alpha$, and
   $\delta$ is not a loop.
  Then there are arrows 
  $\alpha'$ and $\beta'$ with $\alpha - \alpha' \in (\rad A)^2$
  and $\beta - \beta' \in (\rad A)^2$
  such that $(\alpha')^2$ lies in $\soc (e_i A)$, 
  $\delta \beta' = 0$, and $\alpha' A \cap \beta' A \subseteq  \soc (e_i A)$.
\end{enumerate}
\end{lemma}

\bigskip

We have also the following lemma from \cite[IV.1.4.5]{E5}
for loops fixed by $g$.

\begin{lemma}
\label{lem:9.11}
Assume $\alpha$ is a loop at $i$ in $Q_1$ fixed by $g$.
Then, for any choice of arrows $\gamma$ ending at $i$
and $\beta$ starting at $i$, one has $\gamma \alpha = 0$
and $\alpha \beta = 0$.
\end{lemma}

In Lemma \ref{lem:9.10} we have $R_{\gamma} = \alpha'A$ and 
$\alpha'A/S_i \cong \alpha A/S_i= U_i$. One would like to know when this
necessarily implies that $\alpha'A \cong \alpha A$.

\begin{lemma}
\label{lem:9.12}
Let $\alpha$ be an arrow in $Q$ 
with $j = t(\alpha)$ a $2$-vertex, 
$i = s(\alpha)$, 
and $\alpha A / S_i = U_i$.
Then $\dim_K \Ext_A^1(U_i,S_i) = 2$
if $\Omega_A(\alpha A) = \gamma A$ for an arrow $\gamma$ in $Q$
from $j$ to $i$,
and $\dim_K \Ext_A^1(U_i,S_i) = 1$
otherwise.
\end{lemma}

\begin{proof}
There is a commutative diagram in $\mod A$ with exact rows
\[
 \xymatrix{
   0 \ar[r] & \Omega_A(\alpha A) \ar[d] \ar[r] 
     & P_j \ar[d]^{\id} \ar[r] & \alpha A \ar[d]^{p} \ar[r] & 0 
  \\
   0 \ar[r] & \Omega_A(U_i) \ar[r] & P_j \ar[r] & U_i \ar[r] & 0 
  }
\]
with $p$ the canonical epimorphism, and hence a short exact sequence
of the form
\[
   0 \rightarrow
   \Omega_A(\alpha A)  \rightarrow
   \Omega_A(U_i) \rightarrow 
   S_i \rightarrow 
   0. 
\]
We also note that $\Omega_A(\alpha A) \cong R_{\alpha}$
has a simple top and must be generated by an arrow $\rho$ starting at $j$. 
Applying 
$\Hom_A(-,S_i)$
to the lower exact sequence of the above diagram
we obtain an exact sequence of $K$-vector spaces
\[
   0 \rightarrow
   \Hom_A(U_i,S_i) \rightarrow
   \Hom_A(P_j,S_i) \rightarrow
   \Hom_A\big(\Omega_A(U_i),S_i\big) \rightarrow
   \Ext_A^1(U_i,S_i) \rightarrow
   0, 
\]
and hence an isomorphism of $K$-vector spaces
\[
   \Ext_A^1(U_i,S_i) \cong
   \Hom_A\big(\Omega_A(U_i),S_i\big) , 
\]
because 
$\topp(U_i) = \topp(P_j)$.
Further, we see that
$\dim_K \Hom_A(\Omega_A(U_i),S_i) = 1$
if the top of $\Omega_A(\alpha A)$ is not isomorphic to $S_i$, and
$\dim_K \Hom_A(\Omega_A(U_i),S_i) = 2$
if $S_i$ is the top of $\Omega_A(\alpha A)$.
But $S_i$ is the top of $\Omega_A(\alpha A)$
if and only if 
$\Omega_A(\alpha A) = \rho A$ for an arrow $\rho$ in $Q$
from $j$ to $i$.
This proves the claim.
\end{proof}

\begin{corollary}
\label{cor:9.13}
Let $\alpha$ be an arrow in $Q$ from $i$ to $j$
and suppose there is no arrow in $Q$ from $j$ to $i$.
Then $\alpha A \cong \alpha' A$
for any arrow $\alpha'$ such that $\alpha A/S_i \cong \alpha' A/S_i$.
\end{corollary}

\begin{proof}
It follows from Lemma~\ref{lem:9.12}
that $\Ext_A^1(U_i,S_i)$ is one-dimensional,
where $U_i = \alpha A/ S_i$.
Hence any two indecomposable modules which are
extensions of $U_i$ by $S_i$ are isomorphic.
\end{proof}


Next, we want to show that $A$ is special biserial,
with the exception as described in Proposition~\ref{prop:9.15} . 
The following lemma
which is slightly more general, will be used several times for the proof.

For a vertex $i \in Q_0$, $u \in e_i A e_i$ is called a
\emph{normalized unit} if $u - e_i \in \rad e_i A e_i$.

\begin{lemma}\label{lemma:9.14}
 Assume 
$Q$ has a subquiver 
\[
  k \xrightarrow{\alpha} 
  j \xrightarrow{\gamma} 
  i \xrightarrow{\delta} l
 .
\]
Assume that
\begin{enumerate}[(i)]
 \item
 $\Omega_A(\delta A) = \gamma A$ and $\Omega_A(\alpha A) = \delta'A$,
in particular $\delta \gamma = 0$ and $\alpha\delta' = 0$.
 \item
 $\delta A \cong \delta' A$.
\end{enumerate}
Then there are normalized units  $u, v\in e_iAe_i$ such that
$\delta = \delta vu$ and $\delta' = \delta'uv$, and moreover
$(\delta'u)\gamma = 0$
and $\alpha(\delta'u) = 0$.
In particular, $\delta' - \delta'u \in (\rad A)^2$.
\end{lemma}

\begin{proof}
Let $\psi: \delta A \to \delta'A$ be an isomorphism
in $\mod A$. 
We have
\[\psi(\delta) = \delta'u, \ \ \psi^{-1}(\delta') = \delta v,\]
for some $u,v\in e_iAe_i$.
Then
\[
  \delta = \psi^{-1}\psi(\delta) = \psi^{-1}(\delta'u) = \psi^{-1}(\delta')u
= \delta vu,
\]
and similarly $\delta' = \delta'uv$.
Hence $\delta(e_i-vu) = 0$ and therefore
$(e_i-vu) \in \Omega(\delta A) = \gamma A$.
Write $e_i-vu = \gamma z$ for some $z\in A$. Hence $e_i-vu$ lies
in  the radical of $e_iAe_i$ and  is therefore nilpotent. This
implies that  $vu$ is a unit in
$e_iAe_i$ and so are $u$ and $v$. We may assume that $u, v$ are normalized, so that $u$ is of the form $e_i + u'$ with $u'$ in the radical. 
Clearly $\alpha(\delta'u) = 0$. As well
$0 = \psi(0) = \psi(\delta\gamma) = \psi(\delta)\gamma = (\delta'u)\gamma.$
Since $u$ is normalized, we have $\delta' - \delta'u \in (\rad A)^2$.
\end{proof}

\begin{proposition} \label{prop:9.15} 
The algebra $A$ is special biserial, except possibly 
that squares of loops may be non-zero in
the socle.
\end{proposition}

\begin{proof}
(I) \ We prove this first when $Q$ is not the Markov quiver 
(see Example~\ref{ex:4.4}).

\smallskip
 First we  show that for suitable choice of arrows, the condition
on paths of length two of the definition holds.
That is, we must show that for suitable choice, the product
of two arrows  along a cycle of $f$, 
in Notation~\ref{notation:9.9}, is zero.

(1) \ Assume $i\in I_0$, then the 
cycle of $f$ containing  $i$ clearly has
length three.
Then, by 
Proposition~\ref{prop:9.3},
we may assume that it is of the form
\[
   (i\,\gamma\,\delta)
\]
and $\gamma\delta = 0$.

\smallskip
From now, we need to consider only arrows adjacent to 2-vertices.

\smallskip

(2) Consider a fixed point $\alpha \in Q_1$ of $f$. Then $\alpha$ is a loop. 
We have
the setting as in part (ii) of 
Lemma~\ref{lem:9.10}, and 
we may assume that
$\alpha^2$ lies in the socle of $A$.
Then we can write
\[
   \alpha^2 = b\omega_i
\]
where $\omega_i$ generates the socle of $P_i = e_i A$ and $b \in K$.
Moreover, we see from this directly that 
$\Omega_A^2(\alpha A) \cong \alpha A$.

\smallskip

(3)\ 
Consider a loop $\gamma$ which is not fixed
by $f$. With the notation of Lemma~\ref{lem:9.10}(ii), 
one of $\alpha, \beta$ is equal to $\gamma$ and
$\alpha \neq \gamma$ and therefore $\beta = \gamma$.

Since $\gamma$ is fixed by $g$, we have by Lemma \ref{lem:9.11} that
$\gamma\alpha = 0$, and $\Omega(\gamma A) \cong \alpha A$. 
Let $j=t(\alpha)$. 
This is a 2-vertex $\neq i$
 and $\Omega(\alpha A)$ is generated
by an arrow which must end at $i$ since the period is $\leq 3$. It 
must be an arrow $\delta'$ with $\alpha \delta' = 0$ and 
$\delta'-\delta \in (\rad A)^2$. Now 
again by Lemma \ref{lem:9.11} we have even $\delta\gamma =0$
and $\Omega(\alpha A) \cong \Omega^{-1}(\gamma A) \cong  \delta A$.

Now we have the hypotheses of  Lemma \ref{lemma:9.14}  which shows that
we can assume also that $\alpha \delta = 0$.

\bigskip

We are left to consider cycles of $f$ on arrows which do not
contain any loops.

\smallskip

(4) \ Consider such a cycle of length three, say $(\gamma \ \alpha \ \delta)$.
This must then pass through
three distinct
vertices. By assumption, $Q$ is not the Markov quiver, so at most one of the
three arrows can be  part of a double arrow. 
So we may label such that 
$\alpha, \delta$ are not double arrows. 
Then we may apply part (i) of Lemma \ref{lem:9.10} which gives that
we may replace $\alpha$ (and $\beta$ with the labelling there)
 and assume $\gamma\alpha = 0$ and $\Omega(\gamma A) = \alpha A$.
Say $\alpha$ ends at $j$ so that $\delta$ starts at $j$. 
Then $j$ must be
a 2-vertex and $\delta$ is not part of a double arrow. 
So we may assume
$\alpha\delta = 0$ and $\Omega(\alpha A) = \delta A$. 

The period is three, so $\Omega(\delta A) \cong \gamma A$ and $\Omega(\delta A)
= \gamma' A$ for an arrow $\gamma'$ starting at $t(\delta)$. 
Now we use Lemma \ref{lemma:9.14} again. This shows that we may assume
$\delta\gamma =0$.

(5) \
Assume now that $\alpha$ is not a loop and $\alpha A$
has $\Omega_A$-period $2$. 
Then, by part (3) in Notation~\ref{notation:9.9}, 
we have 
$\Omega_A(\alpha A) = \beta A$ for an arrow $\beta$ 
with $\alpha\beta = 0$.
Then we must have 
$\Omega_A(\beta A) \cong \alpha A$ and also $\Omega(\beta A) = \alpha'A$ for
some arrow $\alpha': i\to j$. 

We apply the Lemma \ref{lemma:9.14} again, this
 gives that we may assume $\alpha\beta = 0$ and
$\beta\alpha = 0$. 

Let $\alpha$ be an arrow starting at a $2$-vertex, and let $n_{\alpha}$ be the size of
the $g$-orbit of $\alpha$. Then there is a maximal $m = m_{\alpha} \geq 1$ such that
    $B_{\alpha}: = (\alpha g(\alpha) \ldots g^{n_{\alpha}-1}(\alpha))^m \neq 0$.
This generates the socle of $\alpha A$. The parameters $n_{\alpha}$ and $m_{\alpha}$ are constant on $g$-orbits since $A$ is symmetric. 
Whenever $i\in Q_0$ is a 2-vertex and $\gamma, \delta$ are the arrows starting at $i$, there is a non-zero scalar
such that
$$B_{\gamma}  = c_{\delta}B_{\delta}
$$
The algebra is symmetric, and by a standard argument one can scale some arrows suitably and obtain that any such scalar
$c_{\delta}$ is equal to $1$.

\bigskip

(II) Now we prove the Proposition for an algebra $A$ of generalized dihedral type
 where the quiver is as follows
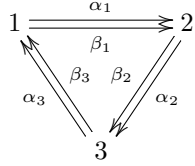
\[
  \xymatrix@R=3.pc@C=1.8pc{
    1
    \ar@<.35ex>[rr]^{\alpha_1}
    \ar@<-.35ex>[rr]_{\beta_1}
    && 2
    \ar@<.35ex>[ld]^{\alpha_2}
    \ar@<-.35ex>[ld]_{\beta_2}
    \\
    & 3
    \ar@<.35ex>[lu]^{\alpha_3}
    \ar@<-.35ex>[lu]_{\beta_3}
  }
\]

We choose arrows $\alpha_i, \beta_i$ 
such that $e_iA = \alpha_iA + \beta_iA$ with $\alpha_iA\cap \beta_iA = \soc (e_iA)$. 
For  this quiver, every module generated by an arrow must have
$\Omega$-period three. 

Then we use Lemma \ref{lem:9.8}, and get identifications,
that is 
$$\{ \alpha_iA/S_i, \beta_iA/S_i\} = \{ R_{\alpha_{i-1}}/S_i, R_{\beta_{i-1}}/S_i\}
$$
(taking $i$ modulo $3$).

(1) We start with $\alpha_1, \beta_1$. We may assume 
$R_{\alpha_1}/S_2 \cong \alpha_2A/S_2$ and
$R_{\beta_1}/S_2 \cong \beta_2A/S_2$. 
By Lemma \ref{lem:9.8} 
we may assume 
$\alpha_2$ generates 
$R_{\alpha_1}$ and 
$\beta_2$ generates $R_{\beta_1}$.

(2) The same reasoning gives that we may assume $R_{\alpha_2} = \alpha_3A$
and $R_{\beta_2} = \beta_3A$. 

So far we have $\Omega(\alpha_iA) = \alpha_{i+1}A$ and
$\Omega(\beta_iA) = \beta_{i+1}A$ for $i \in \{1, 2\}$. 
The period of $\alpha_iA$  is $3$, and therefore
$\Omega(\alpha_3A) \cong \alpha_1A$. 
As well $\Omega(\alpha_3A) = \alpha_1'A$ where $\alpha_1'$ is an arrow
from $1$ to $2$. 
Similarly,
$\beta_1A \cong \Omega(\beta_3A) = \beta_1'A$ for an arrow
$\beta_1'$ from $1$ to $2$. 

We apply Lemma \ref{lem:9.8} which shows that, 
without loss of generality,
$\alpha_1'=\alpha_1$ and $\beta_1' = \beta_1$. 

Now we can get socle relations, similarly as in the general case.
In total the relations are
\begin{align*}
 \alpha_1 \alpha_2 &= 0,
 &
 \beta_1 \beta_2 &= 0,
 &
  (\alpha_1 \beta_2 \alpha_3 \beta_1 \alpha_2 \beta_3)^m
    &=
  (\beta_1 \alpha_2 \beta_3 \alpha_1 \beta_2 \alpha_3)^m
,\\  
 \alpha_2 \alpha_3 &= 0,
 &
 \beta_2 \beta_3 &= 0,
 &
  (\alpha_2 \beta_3 \alpha_1 \beta_2 \alpha_3 \beta_1)^m
    &=
  (\beta_2 \alpha_3 \beta_1 \alpha_2 \beta_3 \alpha_1)^m
,\\  
 \alpha_3 \alpha_1 &= 0,
 &
 \beta_3 \beta_1 &= 0, 
 &
  (\alpha_3 \beta_1 \alpha_2 \beta_3 \alpha_1 \beta_2)^m
    &=
  (\beta_3 \alpha_1 \beta_2 \alpha_3 \beta_1 \alpha_2)^m
.
\end{align*}
Here $m\geq 1$, and this defines the weight function. 
\end{proof}

We will now construct a biserial weighted triangulation algebra such that the
idempotent algebra
associated to an appropriate set of 
2-triangle disks is isomorphic  to $A$.

\medskip

We fix $A$, and then $(Q, f)$ is fixed.
We denote by $\Delta = \Delta(Q, f)$ the family of all 2-cycles 
$\xymatrix{ a \ar@<.5ex>[r]^{\alpha} & b \ar@<.5ex>[l]^{\beta} }$
in $Q_1$ with $\alpha\beta = 0$ and $\beta\alpha = 0$.

We define now a triangulation quiver $(\tilde{Q},\tilde{f})$ with
$\tilde{Q} = (\tilde{Q}_0,\tilde{Q}_1,\tilde{s},\tilde{t})$. This is obtained from $(Q, f)$ as follows.
\begin{enumerate}[(I)]
 \item
  For each vertex $i \in I_0$ and the arrows
  $
  \xymatrix{ i \ar@<.5ex>[r]^{\gamma_i} & i^* \ar@<.5ex>[l]^{\delta_i} }
  $
  adjacent to  $i$, we create in $\tilde{Q}_1$ a loop $\eta_i$ at $i$,
  and set 
  $\tilde{f}(\eta_i) = \gamma_i$,
  $\tilde{f}(\gamma_i) = \delta_i$, and
  $\tilde{f}(\delta_i) = \eta_i$.
 \item
  We replace each $2$-cycle 
  $
    \xymatrix{ a \ar@<.5ex>[r]^{\alpha} & b \ar@<.5ex>[l]^{\beta} }
  $
  in $\Delta$ 
  by a $2$-triangle disk in $\tilde{Q}$ of the form
  
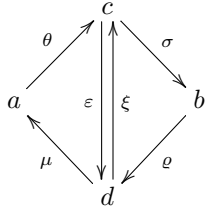
\[
  \xymatrix{
  & 
    c
      \ar@<-.5ex>[dd]_{\varepsilon}
      \ar[rd]^{\sigma}
  \\ 
    a
      \ar[ru]^{\theta}
  &&
    b
      \ar[ld]^{\varrho}
  \\ 
  &
    d
      \ar[lu]^{\mu}
      \ar@<-.5ex>[uu]_{\xi}
  } 
  \]
  with
  $\tilde{f}(\theta) = \varepsilon$,
  $\tilde{f}(\varepsilon) = \mu$,
  $\tilde{f}(\mu) = \theta$,
  $\tilde{f}(\sigma) = \varrho$,
  $\tilde{f}(\varrho) = \xi$,
  $\tilde{f}(\xi) = \sigma$.
 \item
  We keep in $\tilde{Q}_1$ all arrows $\omega$ of $Q_1$
  which do not belong to $2$-cycles in $\Delta$
    and set
  $\tilde{f}(\omega) = f(\omega)$.
\end{enumerate}
We note that $(\tilde{Q},\tilde{f})$ 
is a triangulation quiver with
$|\tilde{Q}_0| = |Q_0| + 2|\Delta|$.
We denote by $\tilde{\Sigma}$ the family 
of all $2$-triangle disks in $(\tilde{Q},\tilde{f})$ 
created from  the $2$-cycles in $\Delta$.
Observe also that the border
$\partial(\tilde{Q},\tilde{f})$ 
of
$(\tilde{Q},\tilde{f})$ 
is given by the sources (targets) of all loops $\nu$
in $Q_1$ with $f(\nu) = \nu$.

We denote by $\tilde{g} : \tilde{Q}_1 \to \tilde{Q}_1$
the permutation induced by $\tilde{f}$,
and by $\cO(\tilde{g})$ the set of all $\tilde{g}$-orbits in $\tilde{Q}_1$.

Let $\tilde{\cO}(\gamma)$ be the $\tilde{g}$-orbit of an arrow $\gamma$ of
$\tilde{Q}$. These orbits are as follows:

\begin{enumerate}[(I)]
\addtocounter{enumi}{3}
 \item
  For each vertex $i \in I_0$, we have the loop $\eta_i$ with 
  $\tilde{g}(\eta_i) = \eta_i$.
 \item
For each $2$-triangle disk created from a cycle in $\Delta$ we have the orbit of size 2 consisting
of $\varepsilon$ and $\xi$,
as in the above subquiver. 

\item For each $2$-triangle disk in $(\tilde{Q},\tilde{f})$ 
  created by a $2$-cycle $(\alpha \, \beta)$ from $\Delta$,
  we have 
  $\tilde{g}^{-1}(\theta) = {g}^{-1}(\alpha)$,
  $\tilde{g}(\theta) = \sigma$,
  $\tilde{g}(\sigma) = g(\alpha)$,
  $\tilde{g}^{-1}(\varrho) = {g}^{-1}(\beta)$,
  $\tilde{g}(\varrho) = \mu$,
  $\tilde{g}(\mu) = g(\beta)$.
That is, we obtain for example the $\tilde{g}$-orbit of
$g^{-1}(\alpha)$ by replacing $\alpha$ with $\theta, \sigma$, and keeping
the rest, and similarly we replace $\beta$ by $\rho, \mu$
to obtain the $\tilde{g}$-orbit of $g^{-1}(\beta)$.
\item
If $\gamma \in Q_1$ and $\cO(\gamma)$ does not contain an arrow from a cycle in $\Delta$ then $\tilde{O}(\gamma) = \cO(\gamma)$.
\end{enumerate}

In particular,  we have
$|\cO(\tilde{g})| = |\cO(g)| + |I_0| + |\Delta|$,
where $\cO(g)$ is the set of all $g$-orbits in $Q_1$.
We also note that two arrows $\gamma$ and $\delta$
in $Q_1 \cap \tilde{Q}_1$ belong to the same $g$-orbit in $Q_1$
if and only if they belong to the same $\tilde{g}$-orbit in $\tilde{Q}_1$.
\medskip

We set $\tilde{n}_{\gamma} = |\tilde{\cO}(\gamma)|$.
For each arrow $\delta \in Q_1$, we
had already defined  $n_{\delta}$ to be the length of the $g$-orbit $\cO(\delta)$
of $\delta$ in $Q_1$.
Clearly, $\tilde{n}_{\delta}  \geq n_{\delta}$ 
for any arrow $\delta \in \tilde{Q}_1 \cap Q_1$, 
but in general we may have
$\tilde{n}_{\delta}  > n_{\delta}$. 

We shall define now a suitable weight function 
$\tilde{m}_{\bullet} : \cO(\tilde{g}) \to \bN^*$.

(1) For an arrow in each of the new orbits of $\tilde{g}$, 
we set $\tilde{m}=1$, that is
$\tilde{m}_{\eta_i} = 1$ and $\tilde{m}_{\varepsilon}=1=\tilde{m}_{\xi}$.

(2) Any other orbit of $\tilde{g}$ contains $\gamma \in Q_1$ and
$\tilde{\cO}(\gamma) \cap Q_1 = \cO(\gamma)$. 
Set $\tilde{m}_{\rho} = m_{\gamma}$ for any arrow
$\rho$ in this $\tilde{g}$-orbit.

The border $\partial(\tilde{Q}, \tilde{f})$  consists of 
the sources of all loops whose square is in the socle. 
From the presentation of $A$ we have that if $\nu$ is a loop with
$\nu^2$ is in the socle then  there is $\tilde{b}_{i} \in K$,
with $i = \tilde{s}(\nu)$, such that
\[
   \nu^2 =  \tilde{b}_{i} \big(\nu g(\nu) \dots g^{n_{\nu}-1}(\nu)\big)^{m_{\nu}}
.
\]
This define a border function
$\tilde{b}_{\bullet} : \partial(\tilde{Q},\tilde{f}) \to K$.

\bigskip

We have $(\tilde{Q}, \tilde{f})$ which is a triangulation quiver. 
We also
have $\tilde{m}_{\bullet}$ and $\tilde{b}_{\bullet}$. 
With these data, let $B = B(\tilde{Q}, \tilde{f}, \tilde{m}_{\bullet}, 
\tilde{b}_{\bullet})$ be the associated socle deformed 
biserial weighted triangulation  algebra, and 
$B (\tilde{Q},\tilde{f},\tilde{\Sigma},\tilde{m}_{\bullet},\tilde{b}_{\bullet})$
be its idempotent algebra with respect to the family $\tilde{\Sigma}$ of $2$-triangle disks
in $(\tilde{Q},\tilde{f})$.
Then it follows from the above that $A$
is isomorphic to
$B (\tilde{Q},\tilde{f},\tilde{\Sigma},\tilde{m}_{\bullet},\tilde{b}_{\bullet})$.
We also note that 
$B (\tilde{Q},\tilde{f},\tilde{\Sigma},\tilde{m}_{\bullet},\tilde{b}_{\bullet})$
is socle equivalent to the idempotent algebra
$B (\tilde{Q},\tilde{f},\tilde{\Sigma},\tilde{m}_{\bullet})$
of the biserial weighted triangulation algebra
$B (\tilde{Q},\tilde{f},\tilde{m}_{\bullet})$.
Moreover, by 
Theorem~\ref{th:4.1},
$(\tilde{Q},\tilde{f})$
is the triangulation quiver $(Q(S,\vv{T}),f)$
associated to a directed triangulated surface $(S,\vv{T})$.
This completes the proof that (i) implies (ii), and hence the proof of 
Theorem~\ref{th:main1}.

\section{Proof of Theorem \ref{th:main2}}\label{sec:proof2}

Let $A$ be an algebra.
For a module $M$ in $\mod A$,
we denote by $[M]$ the image of $M$ in the Grothendieck
group $K_0(A)$ of $A$. Hence, for two modules $M$, $N$
in $\mod A$, we have $[M] = [N]$ if and only if
$M$ and $N$ have the same simple composition factors
including the multiplicities.

\begin{proposition}
\label{prop:10.1}
Let $B = B(Q,f,m_{\bullet})$
be a biserial weighted triangulation algebra
and suppose that the Cartan matrix $C_B$ of $B$ is non-singular.
Then $B$ is an algebra of strict dihedral type.
\end{proposition}

\begin{proof}
By Proposition~\ref{prop:5.1},
we only have to show that $|Q_0|=2$ or $3$ and that the number
of stable tubes of rank $3$ in $\Gamma_B^s$ is $|Q_0|-1$. 
In fact we will show the first part, the second part will follow.

We  apply Theorem \ref{th:main1}.
We know that $A$ is special biserial, and we 
have a permutation $f$ on $Q_1\cup I_0$  describing the zero relations of length two
and 1-vertices of $Q$. This permutation has cycles of length $\leq 3$, and every arrow belongs to a unique cycle. 
In the following, we exploit the exact sequences from 
Notation~\ref{notation:9.9}.

(1) Assume (for a contradiction) that  $f$ has a 2-cycle $(\alpha \ \beta)$ and $\alpha\beta=0$
and $\beta\alpha = 0$. 
Then $j= t(\alpha) \neq s(\alpha)=i$, for otherwise $A$ would be local. 
Then we have an exact sequence
\[
0 \to \alpha A \to P_i \to P_j \to \alpha A\to 0 .
\]
It follows that $[P_i] = [P_j]$ and $C_A$ is singular, a contradiction.
So no such cycle exists.

\bigskip

(2) Consider a $2$-vertex $i$. Suppose there are two $f$-cycles of length three passing through $i$, and let
$\alpha, \ba$ be the arrows starting at $i$. Then we have  exact sequences
\begin{gather*}
0\to \alpha A \to P_i \to P_k \to P_j\to \alpha A\to 0 ,\\
0\to \ba A \to P_i \to P_s  \to P_t \to \ba A\to 0 .
\end{gather*}
Here, the cycle of $\alpha$ passes through $i, j, k$ 
(where two of these vertices may be equal), 
and similarly the cycle of
$\ba$ passes through vertices $i, t, s$. 
Note that $[\alpha A ] + [\ba A] = [P_i]$. 
Hence, if we take the direct sum of these sequences,
we get an identity for composition factors of projective modules, 
namely
\[
  2[P_i] + [P_k] + [P_s] = [P_j] + [P_t] + 2[P_i] ,
\]
and $[P_k] + [P_s] = [P_j] + [P_t]$. Hence
\[
  \{ [P_k], [P_s]\} = \{ [P_j], [P_t] \}
  .
\]

If $k=s$ then also $j=t$.  
But then to have $C_A$ non-singular it follows that
$k=s=j=t$ and 
we get a contradiction with $Q$ being $2$-regular.
In particular, this shows that $Q$ is not the Markov quiver,
considered in Example~\ref{ex:4.4} .

Hence $k\neq s$ and then $j\neq t$. It follows that 
either $k=j$ and $s=t$, or $k=t$ and $s=j$.
In the first case, these $f$-cycles give a 2-regular subquiver with three vertices 
and two loops which is then all of $Q$, and $Q$ is the quiver with three vertices,
considered in Example~\ref{ex:8.3}. 
In the second case, these $f$-cycles give a quiver with three vertices and no loops,
which also is 2-regular and hence all of $Q$, and then 
$Q$  is the quiver with three vertices considered in Example~\ref{ex:8.4}.

\bigskip

(3) \ Assume $i$ is a 2-vertex where both $f$-cycles 
through $i$ contain a vertex in $I_0$, 
then obviously $Q$ has three vertices.
Suppose one of the $f$-cycles through $i$ contains a 1-vertex, say $j$,  
but not the other. 
Then we have exact sequences
\[
 0 \to \alpha A \to P_i \to P_j \to P_j\to \alpha A \to 0
\]
and
\[
 0\to \ba A \to P_i \to P_k \to P_{r}\to \ba A \to 0 .
\]
with $j = t(\alpha)$ and $r = t(\bar{\alpha})$.
Taking the direct sum gives now the identity for composition factors 
of projective modules
\[
 [P_j] + [P_k] = [P_j] + [P_{r}],
\]
and $[P_k] = [P_r]$, which implies $k=r$. That is, the cycle of $\ba$ has a loop and the quiver has three vertices.

\bigskip

Now assume at most one 3-cycle of $f$ passes through any 2-vertex. 
Then since $Q$ is connected, there can only be one 3-cycle 
and otherwise fixed points of $f$, and then $Q$ has at most three vertices.
We can have the quiver with two vertices and one or two loops. 
We note that
the quiver with three vertices and three loops fixed by $f$,
considered in Example~\ref{ex:4.3} ,
is not possible. If so, then $g$ consists of just one cycle which passes 
through each vertex twice, and it follows that
all projectives have the same composition factors, and $C_A$ is singular.

The second statement follows in each possible case, 
by counting the cycles of $f$ of length three.
\end{proof}

\begin{remark} Given conditions (1), (3) and (4) of the definition
of strict dihedral type, then condition (5) implies that (2) holds.
This is proved in \cite{E4}. The final argument in the above proof uses
that (2) implies (5). We note that the classification of algebras of
dihedral type is not sufficient to prove Theorem 2.
\end{remark}

We may complete now the proof of Theorem \ref{th:main2}.
Since every algebra of strict dihedral type is of generalized 
dihedral type and with non-singular Cartan matrix,
the implication 
(i) $\Rightarrow$ (ii) 
holds.
Assume now that $A$ is an algebra of generalized dihedral type
and the Cartan matrix $C_A$ is non-singular.
It follows from 
Theorem \ref{th:main1} 
that $A$ is socle equivalent to an algebra of the form $B(Q,f,m_{\bullet},\Sigma)$.
Then, applying
Proposition \ref{prop:7.3}, 
we conclude that $A$ is isomorphic to an algebra of the form
$B(Q,f,m_{\bullet},\Sigma,b_{\bullet})$.
Since the Cartan matrix $C_A$ is non-singular,
it follows from
Propositions \ref{prop:7.1} and \ref{prop:7.2}
that $\Sigma$ is empty.
Then $A$ is socle equivalent to $B = B(Q,f,m_{\bullet})$, by
Proposition \ref{prop:6.2}.
Moreover, the Cartan matrices $C_A$ and $C_B$ 
coincide, because $A$ and $B$ are socle equivalent
symmetric algebras.
In particular, $C_B$ is non-singular.
Applying Proposition~\ref{prop:10.1}
we conclude that $B$ is an algebra of strict dihedral type.
But then $A$ is an algebra of strict dihedral type,
again because $A$ is socle equivalent to $B$
(see also Proposition~\ref{prop:6.4} and Theorem~\ref{th:8.5}).

\section*{Acknowledgements}

The results of the paper were partially presented 
at the workshops 
``Brauer Graph Algebras'' (Stuttgart, March 2016)
and 
``Representation Theory of Quivers and Finite Dimensional Algebras''
(Oberwolfach, February 2017),
and the conference 
``Advances of Representation Theory of Algebras: Geometry and Homology'' 
(CIRM Marseille-Luminy, September 2017).

\end{document}